\documentclass[11pt]{amsart}
\usepackage{amsmath,amssymb,latexsym,bm,xcolor}

\usepackage[utf8]{inputenc}

\DeclareFontFamily{U}{mathb}{\hyphenchar\font45}
\DeclareFontShape{U}{mathb}{m}{n}{
<-6> mathb5 <6-7> mathb6 <7-8> mathb7
<8-9> mathb8 <9-10> mathb9
<10-12> mathb10 <12-> mathb12
}{}
\DeclareSymbolFont{mathb}{U}{mathb}{m}{n}
\DeclareMathSymbol{\llcurly}{\mathrel}{mathb}{"CE}

\newcommand{\bZ}{\mbox{${\mathbb Z}$}}

\newcommand{\bC}{\mbox{${\mathbb C}$}}

\newcommand{\QQ}{\mbox{${\mathbf Q}_G$}}
\newcommand{\OQGl}[1]{[{\mathcal O}_{{\mathbf Q}_G}(#1)]}
\newcommand{\OQG}[1]{[{\mathcal O}_{{\mathbf Q}_G(#1)}]}
\newcommand{\Fh}{\mathfrak{h}}

\def\mytilde{\kern-.015in\hbox{\lower.03in\hbox{\~{}}}\kern-.01in}
\def\Waf{W_{\mathrm{af}}}
\def\Wafp{W_{\mathrm{af}}^{\ge 0}}

\DeclareMathOperator{\wend}{end}
\DeclareMathOperator{\hgt}{height}
\DeclareMathOperator{\wt}{wt}
\DeclareMathOperator{\sgn}{sgn}
\DeclareMathOperator{\dn}{down}

\newcommand{\casefour}[4]{\left\{ \negthickspace \begin{array}{ll} #1 &\mbox{if $#2$} \\ \\[-12pt] {#3} &\mbox{if $#4$}\end{array} \right.}

\newcommand{\QLS}{\mathop{\rm QLS}\nolimits}
\newcommand{\QB}{\mathop{\rm QB}\nolimits}
\newcommand{\Deg}{\mathop{\rm deg}\nolimits}

\newcommand{\rh}{\widehat{r}}

\newtheorem{thm}{Theorem}

\newtheorem{prop}[thm]{Proposition}

\newtheorem{dfn}[thm]{Definition}

\newtheorem{cor}[thm]{Corollary}
\newtheorem{lem}[thm]{Lemma}
\newtheorem{exa}[thm]{Example}

\newtheorem{rema}[thm]{Remark}

\newtheorem{claim}{Claim}[thm]

\newcommand{\bp}{{\bf p}}
\newcommand{\bq}{{\bf q}}
\newcommand{\cri}[1]{{#1}}
\newcommand{\revise}[1]{{#1}}

\newcommand{\edge}[1]{ \xrightarrow{\hspace{2pt}#1\hspace{2pt}} }
\newcommand{\mcr}[1]{\lfloor #1 \rfloor}
\newcommand{\pair}[2]{\langle #1,\,#2 \rangle}
\newcommand{\ha}[1]{\widehat{#1}}
\newcommand{\ol}[1]{\overline{#1}}
\newcommand{\ti}[1]{\widetilde{#1}}

\newcommand{\SRL}[1]{\mathsf{R}^{[#1]}}

\newcommand{\CA}{\mathcal{A}}
\newcommand{\SQ}{\mathsf{Q}}
\newcommand{\SX}{\mathsf{X}}
\newcommand{\SR}{\mathsf{R}}
\newcommand{\st}{\mathsf{t}}
\newcommand{\be}{\mathbf{e}}
\newcommand{\bk}{\mathbf{k}}
\newcommand{\bl}{\mathbf{l}}
\newcommand{\bG}{\mathbf{G}}
\newcommand{\af}{\mathrm{af}}
\newcommand{\Par}{\mathop{\rm Par}\nolimits}

\newcommand{\bchi}{\bm{\chi}}

\newcommand{\bqed}{\quad \hbox{\rule[-0.5pt]{3pt}{8pt}}}

\DeclareMathOperator{\Ht}{height}
\DeclareMathOperator{\ed}{end}

\newenvironment{enu}{%
 \begin{enumerate}%
}{\end{enumerate}}

\newlength{\cellsize}
\newcommand\tableau[1]{
\vcenter{
\let\\=\cr
\baselineskip=-16000pt
\lineskiplimit=16000pt
\lineskip=0pt
\halign{&\tableaucell{##}\cr#1\crcr}}}

\newcommand{\tableaucell}[1]{{%
\def \arg{#1}\def \void{}%
\ifx \void \arg
\vbox to \cellsize{\vfil \hrule width \cellsize height 0pt}%
\else
\unitlength=\cellsize
\begin{picture}(1,1)
\put(0,0){\makebox(1,1){$#1$}}
\put(0,0){\line(1,0){1}}
\put(0,1){\line(1,0){1}}
\put(0,0){\line(0,1){1}}
\put(1,0){\line(0,1){1}}
\end{picture}%
\fi}}

\makeatletter

\let\choose\@@choose

\makeatother

\usepackage[left=1in, right=1in, top=1in, bottom=1in]{geometry}

\title[A general Chevalley formula for semi-infinite flag manifolds]{A general Chevalley formula for semi-infinite\\ flag manifolds and quantum $K$-theory}

\author[C.~Lenart]{Cristian Lenart}
\address[Cristian Lenart]{Department of Mathematics and Statistics, State University of New York at Albany, 
Albany, NY 12222, U.S.A.}
\email{clenart@albany.edu}

\author[S.~Naito]{Satoshi Naito}
\address[Satoshi Naito]{Department of Mathematics, Tokyo Institute of Technology,
2-12-1 Oh-Okayama, Meguro-ku, Tokyo 152-8551, Japan}
\email{naito@math.titech.ac.jp}

\author[D.~Sagaki]{Daisuke Sagaki}
\address[Daisuke Sagaki]{Department of Mathematics, 
Faculty of Pure and Applied Sciences, University of Tsukuba, 
1-1-1 Tennodai, Tsukuba, Ibaraki 305-8571, Japan}
\email{sagaki@math.tsukuba.ac.jp}

\begin{document}

\begin{abstract} We give a Chevalley formula for an arbitrary weight for the torus-equivariant $K$-group of semi-infinite flag manifolds, which is expressed in terms of the quantum alcove model. As an application, we prove the Chevalley formula for an anti-dominant fundamental weight for the (small) torus-equivariant quantum $K$-theory $QK_{T}(G/B)$ of a (finite-dimensional) flag manifold $G/B$; this has been a longstanding conjecture about the multiplicative structure of $QK_{T}(G/B)$. In type $A_{n-1}$, we prove that the so-called quantum Grothendieck polynomials indeed represent (opposite) Schubert classes in the (non-equivariant) quantum $K$-theory $QK(SL_{n}/B)$; we also obtain very explicit information about the coefficients in the respective Chevalley formula.
\end{abstract}


\maketitle

\section{Introduction}

This paper is concerned with a geometric application of the combinatorial model known as the {\em quantum alcove model}, introduced in \cite{lalgam}. Its precursor, the alcove model of the first author and Postnikov, was used to uniformly describe the  highest weight {\em Kashiwara crystals} of symmetrizable Kac-Moody algebras \cite{LP1}, as well as the {\em Chevalley formula} for the equivariant $K$-theory of a (finite-dimensional) flag manifold $G/B$ \cite{LP}. More generally, the quantum alcove model was used to uniformly describe certain crystals of affine Lie algebras (single-column {\em Kirillov-Reshetikhin crystals}) and {\em Macdonald polynomials} specialized at $t=0$ \cite{lnsumk2,lnsumk3}. The objects of the quantum alcove model (indexing the crystal vertices and the terms of Macdonald polynomials) are paths in the {\em quantum Bruhat graph} on the Weyl group~\cite{BFP}. In this paper we complete the above picture, by extending to the quantum alcove model the geometric application of the alcove model, namely the $K$-theory Chevalley formula. 

To achieve our goal, we need to consider the so-called {\em semi-infinite flag manifold} $\QQ$ associated to a connected, simply-connected simple algebraic group $G$ over $\mathbb{C}$, 
with Borel subgroup $B$ and maximal torus $T \subset B$.
We give a Chevalley formula for an arbitrary weight for the {($T \times {\mathbb C}^*$)}-equivariant $K$-group $K_{T\times{\mathbb C}^*}(\QQ)$ of $\QQ$, which is described in terms of the quantum alcove model.
In~\cite{knsekt} and \cite{nospcf}, the Chevalley formulas for $K_{T\times{\mathbb C}^*}(\QQ)$ were originally given in terms of the {\em quantum LS path model} in the case of a dominant and an anti-dominant weight, respectively.
For a general (neither dominant nor anti-dominant) weight, there is no quantum LS path model, but there is a quantum alcove model.
Hence, in order to obtain a Chevalley formula for an arbitrary weight, we first need to translate the formulas above to the quantum alcove model by using the weight-preserving bijection between the two models given by Propositions~\ref{qam2qls} and~\ref{qam2qls1}. 
Starting from these translated formulas (Theorems~\ref{chevdomqam} and \ref{chevantidomqam}), we prove a Chevalley formula for $K_{T \times {\mathbb C}^{*}}(\QQ)$ (Theorem~\ref{genchev}) for an arbitrary weight, based on the combinatorics of the quantum alcove model. 
Furthermore, by examining this proof based on the Yang-Baxter equation for quantum Bruhat operators, we were able to generalize quantum Yang-Baxter moves for the quantum alcove model associated to a dominant weight (obtained in~\cite{lalurc}) to the case of an arbitrary weight; see \cite{KLN1} and \cite{KLN2}. 
Here we should add that an inverse Chevalley formula, which describes the structure of $K_{T \times {\mathbb C}^{*}}(\QQ)$ as a module over the representation ring of $T \times {\mathbb C}^{*}$, is obtained in \cite{O}, \cite{KNOS}, and \cite{LNOS} by an approach through the nil-DAHA (nil double affine Hecke algebra) in simply-laced types.

The study of the equivariant $K$-group of semi-infinite flag manifolds was started in~\cite{knsekt}. A breakthrough in this study is \cite{kat1} and \cite{kat2} (see also \cite{kat4}), in which Kato established a certain $\mathbb{Z}[P]$-module isomorphism from the (small) $T$-equivariant quantum $K$-theory $QK_{T}(G/B)$ of the finite-dimensional flag manifold $G/B$ onto the $T$-equivariant $K$-group $K_{T}(\QQ)$ of $\QQ$; here $P$ is the weight lattice generated by the fundamental weights $\varpi_{k}$, $k \in I$.
This isomorphism sends each (opposite) Schubert class in $QK_{T}(G/B)$ to the corresponding semi-infinite Schubert class in $K_{T}(\QQ)$; moreover, it respects the quantum multiplication in $QK_{T}(G/B)$ with the line bundle class $[\mathcal{O}_{G/B}(-\varpi_{k})]$
and the tensor product in $K_{T}(\QQ)$ with the line bundle class $\OQGl{w_{\circ} \varpi_{k}}$ for all $k \in I$, where $w_{\circ}$ is the longest element of the Weyl group $W$ of $G$.
Based on this result, a longstanding conjecture in~\cite{LP} on the multiplicative structure of $QK_{T}(G/B)$, i.e., the Chevalley formula (Theorem~\ref{qkchev}) for anti-dominant fundamental weights $- \varpi_{k}$, $k \in I$, for $QK_{T}(G/B)$, {is proved by our anti-dominant Chevalley formula for $K_{T \times \mathbb{C}^{*}}(\QQ)$ under the specialization at $q = 1$.} 
Also, from the anti-dominant Chevalley formula for $QK_{T}(G/B)$, we can deduce a Chevalley formula for anti-dominant fundamental weights $- \varpi_{k}$, $k \in I \setminus J$, for the $T$-equivariant quantum $K$-theory $QK_{T}(G/P_{J})$ of a partial flag manifold $G/P_{J}$, where $P_{J} \supset B$ is the parabolic subgroup corresponding to a subset $J \subset I$, by making use of the $\mathbb{Z}[P]$-algebra surjection from the polynomial version of $QK_{T}(G/B)$ onto that of $QK_{T}(G/P_{J})$ established in~\cite{kat3}; see \cite{KLNS} and \cite{KNS}. 

As another application of our Chevalley formula for $QK_{T}(G/B)$, we can prove an important conjecture in \cite{lamqgp} for the non-equivariant quantum $K$-theory $QK(SL_{n}/B)$ of the flag manifold $SL_{n}/B$ of type $A_{n-1}$ (Theorem~\ref{qgroth}):
the {\em quantum Grothendieck polynomials}, introduced in \cite{lamqgp}, indeed represent (opposite) Schubert classes in $QK(SL_{n}/B)$. In this way, we generalize  the results of~\cite{fgpqsp}, where the {\em quantum Schubert polynomials} are constructed as representatives for (opposite) Schubert classes in the quantum cohomology of $SL_{n}/B$.  
Therefore, we can use quantum Grothendieck polynomials to compute structure constants in $QK(SL_{n}/B)$ with respect to the (opposite) Schubert basis;
in actual calculations, we just need to expand their products in the basis they form, which is done by \cite[Algorithm~3.28]{lamqgp}; see \cite[Example~7.4]{lamqgp}. This is important, since computing even simple products in quantum $K$-theory is notoriously difficult. 
Also, in our recent preprint \cite{NS}, the second and third authors proved a Pieri-type multiplication formula for quantum Grothendieck polynomials (i.e., \cite[Conjecture~6.7]{lamqgp}), which is a vast generalization of the Chevalley formula (or, equivalently, Monk-type multiplication formula) and enables us to compute many structure constants to which the Chevalley formula does not apply.
Finally, still for $QK(SL_{n}/B)$, we obtain very explicit information about the coefficients in the respective Chevalley formula (Theorem~\ref{qk-coeff}, Proposition~\ref{alg-deg}, and Theorem~\ref{maxmindeg}). 

\subsection*{Acknowledgments}

 C.L. was partly supported by the NSF grant DMS-1855592 and the Simons Foundation grant \#584738. 
 S.N. was partly supported by JSPS Grant-in-Aid for Scientific Research (B) 16H03920 and (C) 21K03198. 
 D.S. was partly supported by JSPS Grant-in-Aid for Scientific Research (C) 15K04803 and 19K03415. 
An extended abstract of this work has appeared in the Proceedings of the 33rd international conference on
Formal Power Series and Algebraic Combinatorics~\cite{lnsccf}.

\section{Background on the quantum Bruhat graph and its parabolic version}\label{backg}

\subsection{Root systems}

Let $\mathfrak{g}$ be a finite-dimensional simple Lie algebra over $\mathbb{C}$
with Cartan subalgebra $\mathfrak{h}$. 
Denote by $\Pi^{\vee} = \{ \alpha_{i}^{\vee} \}_{i \in I}$ and 
$\Pi = \{ \alpha_{i} \}_{i \in I}$ the set of simple coroots and 
simple roots of $\mathfrak{g}$, respectively, and set
%
$Q^{\vee} := \sum_{i \in I} \mathbb{Z} \alpha_i^{\vee}$, 
$Q^{\vee,+} := \sum_{i \in I} \mathbb{Z}_{\ge 0} \alpha_i^{\vee}$. 
Let $\Phi$, $\Phi^{+}$, and $\Phi^{-}$ be 
the set of roots, positive roots, and negative roots of $\mathfrak{g}$, respectively, 
with $\theta \in \Phi^{+}$ the highest root for the set $\Phi$ of roots of $\mathfrak{g}$; 
we set $\rho:=(1/2) \sum_{\alpha \in \Phi^{+}} \alpha$. For $\alpha \in \Phi$, we set 
\begin{equation*}
\sgn(\alpha):=
 \begin{cases}
  1 & \text{if $\alpha \in \Phi^{+}$}, \\[1mm]
  -1 & \text{if $\alpha \in \Phi^{-}$}, 
 \end{cases}
\qquad 
|\alpha|:=\sgn(\alpha) \alpha \in \Phi^{+}, 
\end{equation*}
and denote by $\alpha^{\vee}$ the coroot of $\alpha$. 
Also, let $\varpi_{i}$, $i \in I$, denote the fundamental weights for $\mathfrak{g}$, and set
$P:=\sum_{i \in I} \mathbb{Z} \varpi_{i}$ and 
$P^{+} := \sum_{i \in I} \mathbb{Z}_{\ge 0} \varpi_{i}$. 
Let $W := \langle s_{i} \mid i \in I \rangle$ be the (finite) Weyl group of $\mathfrak{g}$, 
where $s_{i}$ is the simple reflection with respect to $\alpha_{i}$ for $i \in I$. 
We denote by $\ell:W \rightarrow \mathbb{Z}_{\ge 0}$ 
the length function on $W$, by $e \in W$ the identity element, 
and by $w_{\circ} \in W$ the longest element. For $\alpha \in \Phi$, 
denote by $s_{\alpha} \in W$ the reflection with respect to $\alpha$; note that $s_{-\alpha}=s_{\alpha}$. 

Let $J$ be a subset of $I$. 
We set $Q_{J} := \sum_{i \in J} \mathbb{Z} \alpha_i$, 
$\Phi_{J} := \Phi \cap Q_{J}$, $\Phi_{J}^{\pm} := \Phi^{\pm} \cap Q_{J}$, 
$\rho_{J}:=(1/2) \sum_{\alpha \in \Phi_{J}^{+}} \alpha$. 
We denote by $W_J:= \langle s_{i} \mid i \in J \rangle$ the parabolic subgroup of $W$ corresponding to $J$, 
and we identify $W/W_J$ with the corresponding set of minimal coset representatives, denoted by $W^J$; 
note that if $J=\emptyset$, then $W^{J}=W^{\emptyset}$ is identical to $W$. 
For $w \in W$, we denote by $\mcr{w}=\mcr{w}^{J} \in W^{J}$ 
the minimal coset representative for the coset $w W_{J}$ in $W/W_{J}$.

Let $W_{\mathrm{af}} := 
\langle s_{i} \mid i \in I_{\mathrm{af}} \rangle$, with $I_{\mathrm{af}} := I \sqcup \{0\}$, be the (affine) Weyl group 
of the untwisted affine Lie algebra $\mathfrak{g}_{\mathrm{af}}$ associated to $\mathfrak{g}$. 
For each $\xi \in Q^{\vee}$, let $t_{\xi} \in \Waf$ denote the translation 
by $\xi$ (see \cite[Section~6.5]{Kac}). Then, $\bigl\{ t_{\xi} \mid \xi \in Q^{\vee} \bigr\}$ forms 
an abelian normal subgroup of $\Waf$, in which $t_{\xi} t_{\zeta} = t_{\xi + \zeta}$ 
holds for $\xi,\,\zeta \in Q^{\vee}$. Moreover, we know from \cite[Proposition 6.5]{Kac} that
\begin{equation*}
\Waf \cong 
 W \ltimes \bigl\{ t_{\xi} \mid \xi \in Q^{\vee} \bigr\} \cong W \ltimes Q^{\vee}; 
\end{equation*}
note that $s_{0}=s_{\theta}t_{-\theta^{\vee}}$. We set 
$\Wafp := W \times Q^{\vee,+}$, which is a subset of $\Waf$. 


\subsection{The quantum Bruhat graph}


We take and fix a subset $J$ of $I$. 

\begin{dfn}
The (parabolic) quantum Bruhat graph $\QB(W^{J})$ is 
the $(\Phi^{+} \setminus \Phi_{J}^{+})$-labeled
directed graph whose vertices are the elements of $W^{J}$, and 
whose directed edges are of the form{\rm :} $w \edge{\beta} v$ 
for $w,v \in W^{J}$ and $\beta \in \Phi^{+} \setminus \Phi_{J}^{+}$ 
such that $v= \mcr{ws_{\beta}}$, and such that either of 
the following holds{\rm :} 
{\rm (i)} $\ell(v) = \ell (w) + 1${\rm ;} 
{\rm (ii)} $\ell(v) = \ell (w) + 1 - 2 \pair{\rho-\rho_{J}}{\beta^{\vee}}$.
An edge satisfying {\rm (i)} (resp., {\rm (ii)}) is called a Bruhat (resp., quantum) edge. 
\end{dfn}

When $J=\emptyset$, we write $\QB(W)$ for $\QB(W^{\emptyset})$; 
note that in this case, $W^{\emptyset}=W$, $\Phi_{\emptyset}^{+}=\emptyset$, 
$\rho_{\emptyset}=0$, and that $\mcr{w}=w$ for all $w \in W$. The quantum Bruhat graph $\QB(W)$ originates in the Chevalley formula for the quantum cohomology of flag manifolds~\cite{fawqps}.
%
%
\begin{rema}[{see \cite[Remark~6.13]{lnsumk1}}] \label{rem:PQBG}
For each $v,\,w \in W^{J}$, 
there exists a directed path in $\QB(W^{J})$ from $v$ to $w$.
\end{rema}

For a directed path 
$\bp: v=
 v_{0} \edge{\beta_{1}}
 v_{1} \edge{\beta_{2}} \cdots 
       \edge{\beta_{l}}
 v_{l}=w$
in $\QB(W^{J})$, we define the weight $\wt^{J}(\bp)$ of $\bp$ by
\begin{equation*}
\wt^{J}(\bp) := \sum_{
 \begin{subarray}{c}
 1 \le k \le l\,; \\
 \text{$v_{k-1} \edge{\beta_{k}} v_{k}$ is} \\
 \text{a quantum edge}
 \end{subarray}}
\beta_{k}^{\vee} \in Q^{\vee,+}; 
\end{equation*}
when $J=\emptyset$, we write $\wt(\bp)$ for $\wt^{\emptyset}(\bp)$. 
We know the following from \cite[Proposition 8.1]{lnsumk1}. 
%
%
\begin{prop} \label{prop:shortest}
Let $v,\,w \in W^{J}$. If $\bp$ and $\bq$ are 
shortest directed paths in $\QB(W^{J})$ from $v$ to $w$, then  
$\wt^{J}(\bp) \equiv \wt^{J}(\bq)$ modulo $Q_{J}^{\vee}$. 
In particular, if $J=\emptyset$, then $\wt(\bp)=\wt(\bq)$. 
\end{prop}

For $v,\,w \in W^{J}$, we denote by $\ell^{J}(v \Rightarrow w)$ 
the length of a shortest directed path in $\QB(W^{J})$ from $w$ to $v$. 
When $J=\emptyset$, we write $\ell(v \Rightarrow w)$ for $\ell^{\emptyset}(v \Rightarrow w)$. 


Assume that $J=\emptyset$. In this case, we denote by 
$\wt(w \Rightarrow v)$ the weight $\wt(\bp)$ of a shortest directed path 
in $\QB(W)$ from $w$ to $v$, which is independent of the choice of 
a shortest directed path by Proposition~\ref{prop:shortest}. Also, 
we will use the {\em shellability} of the quantum Bruhat graph $\QB(W)$ 
with respect to a {\em reflection order} on the positive 
roots~\cite{Dyer}, which we now recall. 
%
%
\begin{thm}[\cite{BFP}] \label{thm:shell} 
Fix a reflection order on $\Phi^+$.
\begin{enumerate}
\item For any pair of elements $v,w\in W$, there is a unique directed path from $v$ to $w$ in the quantum 
Bruhat graph $\QB(W)$ such that its sequence of edge labels is strictly increasing (resp., decreasing) 
with respect to the reflection order.
\item The path in {\rm (1)} has the smallest possible length $\ell(v\Rightarrow w)$.
\end{enumerate}
\end{thm}

\subsection{Additional results}



In this subsection, we fix a dominant weight 
$\lambda \in P^{+}$, and set 
$J=J_{\lambda}:= 
\bigl\{ i \in I \mid \pair{\lambda}{\alpha_{i}^{\vee}}=0 \bigr\} \subset I$. 
Let $v,w \in W^{J}$, and let $\bp$ be a shortest directed path in $\QB(W^{J})$ 
from $v$ to $w$. Then we deduce by Proposition~\ref{prop:shortest} that 
$\pair{\lambda}{\wt^{J}(\bp)}$ does not depend on the choice of a shortest directed path $\bp$. 
We write $\pair{\lambda}{\wt^{J}(v \Rightarrow w)}$ for $\pair{\lambda}{\wt^{J}(\bp)}$. 
%
%
\begin{lem}[{\cite[Lemma~7.2]{lnsumk2}}] \label{relwt} 
Keep the notation and setting above. 
Let $\sigma,\,\tau \in W^J$. Then, 
$\pair{\lambda}{\wt^{J}(\sigma \Rightarrow \tau)} = 
 \pair{\lambda}{\wt(v \Rightarrow w)}$ 
for all $v \in \sigma W_J$, $w\in\tau W_J$. 
\end{lem}
%
%
\begin{dfn} \label{dfn:QBa}
For a rational number $b \in \mathbb{Q}$, 
we define $\QB_{b\lambda}(W^{J})$ (resp., $\QB_{b\lambda}(W)$) 
to be the subgraph of $\QB(W^{J})$ (resp., $\QB(W)$) 
with the same vertex set but having only those directed edges of 
the form $w \edge{\beta} v$ for which 
$b \pair{\lambda}{\beta^{\vee}} \in \mathbb{Z}$ holds.
\end{dfn}

\begin{lem}[{\cite[Lemma~6.2]{lnsumk2}}] \label{q2ls} 
Keep the notation and setting above. 
Let $w \edge{\gamma} ws_\gamma$ be an edge in $\QB_{b\lambda}(W)$ 
for some rational number $b$. Then there exists a directed path 
from $\mcr{w}$ to $\mcr{w s_\gamma}$ in $\QB_{b\lambda}(W^J)$ 
(possibly of length $0$).
\end{lem}

\begin{lem}[{\cite[Lemma~6.7]{lnsumk2}}] \label{bqbpaths} 
Consider two directed paths in $\QB(W)$ between some $w$ and $v$. Assume that the first one is a shortest 
path, while the second one is in $\QB_{b\lambda}(W)$, for some rational number $b$. 
Then the first path is in $\QB_{b\lambda}(W)$ as well.
\end{lem}

We now recall \cite[Proposition~7.2]{lnsumk1}, which constructs the analogue of (one version of) the so-called {\em Deodhar lifts} \cite{Deo} for the quantum Bruhat graph; we will call them {\em quantum right Deodhar lifts}.

\begin{prop}[\cite{lnsumk1}]\label{mind} Given $v,w\in W$, there exists a unique element $x \in vW_{J}$ such that 
$\ell(w \Rightarrow x)$ attains its minimum value as a function 
of $x\in vW_J$.
\end{prop}

We refer also to \cite[Theorem~7.1]{lnsumk1}, stating 
that the mentioned minimum is, in fact, attained by
the minimum of the coset $vW_J$ with respect to 
the {\em $w$-tilted Bruhat order} $\preceq_w$ on $W$ (see \cite{BFP}). 
Therefore, it makes sense to denote it by $\min(vW_J,\preceq_w)$,
 although we will not use this stronger result. 

The quantum Bruhat graph analogue of the second version of the Deodhar lifts was given in \cite[Proposition~2.25]{nospcf}; we will call these {\em quantum left Deodhar lifts}. The mentioned result is stated based on the so-called {\em dual $v$-tilted Bruhat order} $\preceq_v^*$ on $W$, introduced in \cite[Definition~2.24]{nospcf}. It is proved by reduction to~\cite[Theorem~7.1]{lnsumk1}.

\begin{prop}[\cite{nospcf}]\label{lDl} Given $v,w\in W$, the coset $wW_J$ has a unique maximal element with respect to $\preceq_v^*$, which is denoted by~$\max(wW_J,\preceq_v^*)$.\end{prop}

For our purposes, the weaker version of this result, which is stated below, suffices; this is the analogue of Proposition~\ref{mind}.

\begin{prop}\label{maxd}
Given $v,w\in W$, there exists a unique element $x \in wW_{J}$ such that 
$\ell(x \Rightarrow v)$ attains its minimum value as a function 
of $x\in wW_J$.
\end{prop}

The mentioned element is $\max(wW_J,\preceq_v^*)$. In~\cite{lnsumk1} we gave a proof of Proposition~\ref{mind}, i.e., \cite[Proposition~7.2]{lnsumk1}, which is independent of \cite[Theorem~7.1]{lnsumk1}, mentioned above; this proof was based on \cite[Lemmas~7.4,~7.5]{lnsumk1}. Likewise, Proposition~\ref{maxd} can be proved independently of Proposition~\ref{lDl}, as an immediate consequence of the analogues of the mentioned lemmas. These analogues are stated as Lemmas~\ref{lem1}~and~\ref{lem2} in Section~\ref{addshellab}, and are also needed in the proof of Lemma~\ref{L:mainb1} in that section. 

\section{Background on the combinatorial models}

Throughout this section, $\lambda$ is a dominant weight whose stabilizer is the parabolic subgroup $W_J$ of $W$ for a subset $J \subset I$.

\subsection{Quantum LS paths}\label{sec:qls}

\begin{dfn}[\cite{lnsumk2}]\label{defls} A {\em quantum LS path} $\eta\in\QLS(\lambda)$ is given by two sequences
\begin{equation}\label{E:lsste}
(0 = b_1<b_2<b_3<\cdots<b_t<b_{t+1}=1) \,;\;\; (\kappa(\eta)=\sigma_1,\,\sigma_{2},\,\ldots,\,\sigma_t=\iota(\eta))\,,
\end{equation}
where $b_k\in \mathbb{Q}$, $\sigma_k\in W^J$, and there is a directed path in $\QB_{b_k\lambda}(W^J)$ from $\sigma_{k-1}$ to $\sigma_k$, for each $k=2,\ldots,t$. The elements $\sigma_k$ are called the {\em directions} of $\eta$, while $\iota(\eta)$ and $\kappa(\eta)$ are the {\em initial} and {\em final directions}, respectively. 
\end{dfn}

This data encodes the sequence of vectors 
\begin{equation}\label{seqvec} u_t:=(b_{t+1}-b_t)\sigma_t\lambda\,, \;\:\ldots\,,\;\: u_2:=(b_3-b_2)\sigma_2\lambda\,,\;\: u_1:=(b_2-b_1)\sigma_1\lambda\,. 
\end{equation}
 We can view the quantum LS path $\eta$ as a piecewise-linear 
path given by the sequence of points 
\[0\,,\;\: u_t\,,\;\: u_{t-1}+u_t\,,\;\: \dotsc\,,\;\: u_1+\cdots+u_t\,.\]
There is also a standard way to express $\eta$ as a map $\eta\::\:[0,1]\rightarrow {\mathfrak h}_{\mathbb R}^*$ with $\eta(0)=0$ (where ${\mathfrak h}_{\mathbb R}^* = \mathbb{R} \otimes_{\mathbb{Z}} X$ is the real part of the dual Cartan subalgebra), but we do not need this here.
The endpoint of the path, also called its weight, is $\wt(\eta):=\eta(1)=u_1+\cdots+u_t$. 

We define the {\em (tail) degree function} (cf. \cite[Corollary~4.8]{lnsumk2}) by
\begin{equation}\label{eq:deg}
\Deg(\eta):=-\sum_{k=2}^t(1-b_k) \langle{\lambda},{\wt_J({\sigma}_{k-1} \Rightarrow {\sigma}_{k})}\rangle\,.
\end{equation}

Given $w\in W$, we define $\iota({\eta},{w}) \in W$, called the {\em initial direction of $\eta$ with respect to $w$}, 
by the following recursive formula: 
\begin{equation} \label{eq:tiw}
\begin{cases}
w_0:=w\,, & \\[2mm]
w_k:=\min(\sigma_kW_J,\preceq_{w_{k-1}}) & \text{for $k=1,\ldots,t$}\,, \\[2mm]
\iota({\eta},{w}):=w_t\,. 
\end{cases}
\end{equation}
Also, we set
\begin{equation} \label{eq:xiv}
\xi(\eta,w):= 
\sum_{k=1}^{t} \wt ({w}_{k-1} \Rightarrow {w}_{k})\,, 
\end{equation}
and
\begin{equation} \label{eq:degx}
\Deg_{w}(\eta):= - \sum_{k=1}^{t} (1-b_k) \langle{\lambda},{\wt({w}_{k-1} \Rightarrow {w}_{k})}\rangle\,.
\end{equation}

Given $v\in W$, we define $\kappa({\eta},{v}) \in W$, called the {\em final direction of $\eta$ with respect to $v$}, 
by the following recursive formula: 
\begin{equation} \label{eq:tiw1}
\begin{cases}
v_{t+1}:=v\,, & \\[2mm]
v_k:=\max(\sigma_kW_J,\preceq_{v_{k+1}}^*) & \text{for $k=1,\ldots,t$}\,, \\[2mm]
\kappa({\eta},{v}):=v_1\,. 
\end{cases}
\end{equation}
Also, we set
\begin{equation} \label{eq:xiv1}
\zeta(\eta,v):= 
\sum_{k=1}^{t} \wt ({v}_{k} \Rightarrow {v}_{k+1})\,.
\end{equation}

\subsection{The quantum alcove model}\label{sec:qam}

We say that two alcoves are {adjacent} if they are distinct and have a common wall. Given a pair
of adjacent alcoves $A$ and $B$, we write $A \xrightarrow{\beta}  B$  for $\beta\in\Phi$ if the
common wall is orthogonal to $\beta$ and $\beta$ points in the direction from $A$ to $B$. Recall that
alcoves are separated by hyperplanes of the form
\[
	H_{\beta,l}=\{\mu\in \Fh^{\ast}_{\mathbb{R}} \mid \langle \mu,\beta^\vee \rangle=l\}\,.
\]
We denote by $s_{\beta,l}$ the affine reflection in this hyperplane.

The fundamental alcove is defined as
\[
	A_{\circ} = \{ \mu \in \Fh_{\mathbb{R}}^{\ast} \mid 0< \langle \mu,\alpha^\vee \rangle < 1 \quad \text{for all $\alpha\in \Phi^+$}\}\;.
\]

\begin{dfn}[\cite{LP}]
	An  \emph{alcove path} is a sequence of alcoves $(A_0, A_1, \ldots, A_m)$ such that
	$A_{j-1}$ and $A_j$ are adjacent, for $j=1,\ldots, m.$ We say that $(A_0, A_1, \ldots, A_m)$  
	is \emph{reduced} if it has minimal length among all alcove paths from $A_0$ to $A_m$.
\end{dfn}

Let $\lambda$ be any weight, and
$A_{\lambda}=A_{\circ}+\lambda$ the translation of the fundamental alcove $A_{\circ}$ by the weight $\lambda$.
	
\begin{dfn}[\cite{LP}]\label{deflch}
	The sequence of roots $\Gamma(\lambda)=(\beta_1, \beta_2, \dots, \beta_m)$ is called a
	\emph{$\lambda$-chain (of roots)}, respectively \emph{reduced $\lambda$-chain},  if 
	\[	
		A_0=A_{\circ} \xrightarrow{-\beta_1}  A_1
		\xrightarrow{-\beta_2} \cdots 
		\xrightarrow{-\beta_m}  A_m=A_{-\lambda}
	\]
is an alcove path, respectively reduced alcove path.
\end{dfn}

A reduced alcove path $(A_0=A_{\circ},A_1,\ldots,A_m=A_{-\lambda})$ defines a total order on the hyperplanes, to be called 
{\em $\lambda$-hyperplanes}, which separate $A_\circ$ from $A_{-\lambda}$. This total order is given by the sequence 
$H_{\beta_i,-l_i}$ for $i=1,\ldots,m$, where $H_{\beta_i,-l_i}$ contains the common wall of 
$A_{i-1}$ and $A_i$. Note that $\langle\lambda,\beta_i^\vee\rangle\ge0$, and that the integers $l_i$, called {\em heights}, have the following ranges:
\begin{equation}\label{ranges}0\le l_i\le\langle\lambda,\beta_i^\vee\rangle-1\;\;\mbox{if}\;\;\beta_i\in\Phi^+\,, \;\;\;\;\mbox{and}\;\;\;\;1\le l_i\le\langle\lambda,\beta_i^\vee\rangle\;\;\mbox{if}\;\;\beta_i\in\Phi^-\,.\end{equation}
 Note also that a reduced $\lambda$-chain $(\beta_1, \ldots, \beta_m)$ determines 
the corresponding reduced alcove path, and hence we can identify them as well. 

\begin{rema} {\rm 
An alcove path corresponds to the choice of a word for an element of the affine Weyl group $\Waf^{\prime} \cong W \ltimes Q$ (corresponding to the Langlands dual $\mathfrak{g}^{\vee}$ of $\mathfrak{g}$)} sending $A_\circ$ to $A_{-\lambda}$ \cite[Lemma 5.3]{LP}. For $\lambda$ dominant, another equivalent definition of a reduced alcove path/$\lambda$-chain, based on a root interlacing condition which generalizes a similar condition characterizing reflection orders, can be found in~\cite[Definition~4.1, Proposition~10.2]{LP1}.
\end{rema}

When $\lambda$ is dominant, we have a special choice of a reduced $\lambda$-chain in~\cite[Section 4]{LP1}, which we now recall. 

\begin{prop}[\cite{LP1}]
\label{speciallch} 
Given a total order $I=\{1<2<\dotsm<r\}$ on the set of Dynkin nodes, one may express 
a coroot $\beta^\vee=\sum_{i=1}^r c_i \alpha_i^\vee$
in the ${\mathbb Z}$-basis of simple coroots. Consider the
total order on the set of $\lambda$-hyperplanes defined by 
the lexicographic order on their images in ${\mathbb Q}^{r+1}$ under the map
\begin{equation}
\label{E:stdvec}
	H_{\beta,-l}\mapsto \frac{1}{\langle\lambda, \beta^\vee \rangle} (l,c_1,\ldots,c_r).
\end{equation}
This map is injective, thereby endowing 
the set of $\lambda$-hyperplanes with a total order, which is a reduced $\lambda$-chain. We call it the
{\em lexicographic (lex) $\lambda$-chain}, and denote it by $\Gamma_{\rm lex}(\lambda)$. 
\end{prop}

The rational number $l/\langle\lambda, \beta^\vee \rangle$ is called the {\em relative height} of the $\lambda$-hyperplane $H_{\beta,-l}$. By definition, the sequence of relative heights in the lex $\lambda$-chain is weakly increasing. 

The objects of the quantum alcove model are defined next. This model was introduced in \cite{lalgam} and then used in \cite{lnsumk2,lnsumk3} in connection with Kirillov-Reshetikhin crystals and Macdonald polynomials specialized at $t=0$. Here we consider a generalization of it, by letting $\lambda$ be any weight, as opposed to only a dominant weight, as originally considered; another aspect of the generalization is making the model depend on a fixed element  $w\in W$, such that the initial model corresponds to $w$ being the identity element $e$. In addition to $w$, we fix an arbitrary $\lambda$-chain $\Gamma(\lambda)=(\beta_1,\,\ldots,\,\beta_m)$, and set $r_i:=s_{\beta_i}$, $\widehat{r}_i:=s_{\beta_i,-l_i}$. 

\begin{dfn}[\cite{lalgam}]
\label{def:admissible}
	A subset 
	$A=\left\{ j_1 < j_2 < \cdots < j_s \right\}$ of $[m]:=\{1,\ldots,m\}$ (possibly empty)
 	is a $w$-\emph{admissible subset} if
	we have the following directed path in the quantum Bruhat graph $\QB(W)${\rm :}
	\begin{equation}
	\label{eqn:admissible}
	 \Pi(w,A):\;\;\;\;w\xrightarrow{|\beta_{j_1}|} w r_{{j_1}} 
	\xrightarrow{|\beta_{j_2}|}  wr_{{j_1}}r_{{j_2}} 
	\xrightarrow{|\beta_{j_3}|}  \cdots 
	\xrightarrow{|\beta_{j_s}|}  wr_{{j_1}}r_{{j_2}} \cdots r_{{j_s}}=:\wend(w,A)\,.
	\end{equation}
 	We let ${\mathcal A}(w,\Gamma(\lambda))$ be the collection of all $w$-admissible subsets of $[m]$.
\end{dfn}

We now associate several parameters with the pair $(w,A)$. 
The weight of $(w,A)$ is defined by 
	\begin{equation}
	\label{defwta}
	\wt(w,A):=-w\rh_{{j_1}} \cdots 
	         \rh_{{j_s}}(-\lambda)\,.
	\end{equation}

Given the height sequence $(l_1,\ldots,l_m)$ mentioned above, we define 
the complementary height sequence $(\widetilde{l}_1,\ldots,\widetilde{l}_m)$ 
by $\widetilde{l}_i:=\langle\lambda,\beta_i^\vee\rangle-l_i$. 
Given $A=\{j_1<\cdots<j_s\}\in{\mathcal A}(w,\Gamma(\lambda))$, we set 
\begin{equation*}
A^{-}:=\bigl\{j_i \in A \mid 
  wr_{{j_1}} \cdots r_{{j_{i-1}}} > 
 w r_{{j_1}} \cdots r_{{j_{i-1}}}r_{{j_{i}}}
\bigr\}\,;
\end{equation*}
in other words, we record the quantum steps 
in the path $\Pi(w,A)$ given by \eqref{eqn:admissible}. 
We also define
\begin{equation}
\label{defheight}
\dn(w,A):=\sum_{j\in A^-}|\beta_j|^\vee\in Q^{\vee,+}\,,\;\;\;\;\hgt(w,A):=\sum_{j\in A^-}\sgn(\beta_j)\widetilde{l}_j\,.
\end{equation}

For examples, we refer to \cite{lenfmp,lnsumk2}.

\subsection{Galleries}\label{sec:gal} In this section, we recall from \cite[Appendix]{LP} the reformulation of the alcove model in terms of so-called {\em galleries}, which are similar, but not equivalent, to the LS-galleries of Gaussent-Littelmann~\cite{GaLi}. We also extend this concept to the quantum alcove model, as described in Section~\ref{sec:qam}.

\begin{dfn}[\cite{LP}] A {\em gallery} is a sequence 
$\gamma=(F_0,\,A_0,\,F_1,\,A_1,\, F_2,\, \ldots ,\, F_m,\, A_m,\, F_{m+1})$ 
such that $A_0,\ldots,A_m$ are alcoves{\rm ;}
$F_i$ is a codimension $1$ common face of the alcoves $A_{i-1}$ and $A_i$,
for $i=1,\ldots,m${\rm ;} $F_0$ is a vertex of the first alcove $A_0${\rm ;} and 
$F_{m+1}$ is a vertex of the last alcove $A_m$. 
If $F_{m+1}=\{\mu\}$, then the weight $\mu$ is called the {\em weight} of the gallery, and is denoted by ${\rm wt}(\gamma)$.
We say that a gallery is {\em unfolded} if $A_{i-1}\ne A_i$, 
for $i=1,\ldots,m$.  
\label{def:gallery}
\end{dfn}

A $\lambda$-chain $\Gamma(\lambda)$ corresponds to an alcove path from $A_{\circ}$ to $A_{-\lambda}$ (cf. Definition~\ref{deflch}), and thus determines an unfolded gallery 
$$\gamma(\lambda)=(F_0=\{0\},\,A_0=A_\circ,\,F_1,\,A_1,\, F_2,\, \ldots ,\, F_m,\, A_m=A_{-\lambda},\, F_{m+1}=\{-\lambda\})\,;$$
 see \cite[Lemma~18.3]{LP}. We fix such structures. 

We can define several operations on galleries $\gamma=(F_0,\,A_0,\,F_1,\,A_1,\, F_2,\, \ldots ,\, F_m,\, A_m,\, F_{m+1})$. First, we consider the translation $\gamma+\mu$ for a weight $\mu$, and the image $w(\gamma)$ under a Weyl group element $w \in W$. Then, as in \cite[Section~18.1]{LP}, we define the {\it tail-flip operators} $f_i$, for $i=1,\ldots,m$. To this end, let $\widehat{r}_i$ be the affine reflection with respect to the affine hyperplane containing the face $F_i$. The operator $f_i$ sends the gallery 
$\gamma$ to the gallery 
$$f_i(\gamma):=(F_0,\,A_0,\, F_1,\, A_1,\, \ldots,\, A_{i-1},\,
 F_i'=F_i,\,  A_{i}',\, F_{i+1}',\, A_{i+1}',\, \ldots,\,  A_m',\, F_{m+1}')\,,
$$
where $A_j' := \widehat{r}_i(A_j)$ and $F_j':=\widehat{r}_i(F_j)$, for 
$j=i,\dots,m+1$.
In other words, $f_i$ leaves the initial segment of the gallery from $A_0$ to $A_{i-1}$ intact, 
and reflects the remaining tail by $\widehat{r}_i$. Clearly, the operators $f_i$ commute.

Given a subset $A=\left\{ j_1 < j_2 < \cdots < j_s \right\}$ of $[m]$, we associate with it the gallery $\gamma(w,A):=wf_{j_1}\cdots f_{j_s}(\gamma(\lambda))$. For obvious reasons, we call the elements of $A$ {\em folding positions}.

\begin{prop}\label{conc-gal} {\rm (1)} We have  
\[{\rm wt}(w,A)=-{\rm wt}(\gamma(w,A))\,.\]

{\rm (2)} The first alcove of $\gamma(w,A)$ is $w(A_\circ)$, and the last alcove is $v(A_\circ)+{\rm wt}(\gamma(w,A))$, where $v:={\rm end}(w,A)$. 
\end{prop}
\begin{proof}
Part (1) is a slight extension of \cite[Lemma~18.4]{LP}, whose proof is completely similar. The first part of (2) is straightforward. For the second part of (2), assuming first that $w$ is the identity element $e$, we proceed by induction on the cardinality of $A$. The base case $A=\emptyset$ is obvious. Using the above notation, let $A=\{j=j_s\}$, and $\widehat{r}_j=r_j+\mu$, where $r_j$ is the corresponding non-affine reflection. Then the last alcove in $\gamma(e,A)$ is 
\[\widehat{r}_j(A_{-\lambda})=r_j(A_{-\lambda})+\mu=r_j(A_\circ)+r_j(-\lambda)+\mu=r_j(A_\circ)+\widehat{r}_j(-\lambda)=r_j(A_\circ)+{\rm wt}(\gamma(e,A))\,,\]
which verifies the statement. We continue in this way, by adding $j_{s-1}>\cdots>j_1$ to $A$, in this order, and by applying $w$ at the end. 
\end{proof}

\begin{dfn}\label{def-conc} Consider two galleries 
$$\gamma=(F_0,\,A_0,\,F_1,\,\ldots,\,A_m,\,F_{m+1})\,,\;\;\;\;\;\gamma'=(F_0',\,A_0',\,F_1',\,\ldots,\,A_m',\,F_{m+1}')\,,$$ 
such that $F_{m+1}=F_0'$ and $A_m=A_0'$. Under these conditions, their concatenation $\gamma\ast\gamma'$ is defined in the obvious way{\rm :}
\[\gamma\ast\gamma':=(F_0,\,A_0,\,F_1,\,\ldots,\,A_m=A_0',\,F_1',\,\ldots,\,A_m',\,F_{m+1}')\,.\]
\end{dfn}

\subsection{Additional shellability results}\label{addshellab}

In \cite[Section 4.3]{LS}, we constructed a reflection order $<_\lambda$ on $\Phi^+$ which depends on $\lambda$. 
The bottom of the order $<_\lambda$ consists of the roots in $\Phi^+\setminus\Phi_J^+$. For two such roots 
$\alpha$ and $\beta$, define $\alpha<\beta$ whenever 
the hyperplane $H_{(\alpha,0)}$ precedes $H_{(\beta,0)}$ in the lex $\lambda$-chain (see Proposition~\ref{speciallch}).
This forms an \emph{initial section} (see \cite{Dyer}) of $<_\lambda$.
The top of the order $<_\lambda$ consists of the positive roots in $\Phi_J^+$,
and we fix any reflection order for them. We refer to the reflection order $<_\lambda$ throughout.

\begin{rema}\label{remlex}{\rm 
It is not hard to see that, in the lex $\lambda$-chain, the order on the $\lambda$-hyperplanes $H_{\beta,-l}$ with the 
same relative height (not necessarily equal to 0) is given by the order $<_{\lambda}$ on the corresponding roots $\beta$. We will use this fact implicitly below. }
\end{rema}

We recall \cite[Lemma~6.6]{lnsumk2}, which characterizes the quantum right Deodhar lifts in shellability terms.

\begin{lem}[\cite{lnsumk2}] \label{L:mainb} 
Consider $\sigma,\tau\in W^J$ and $w_J\in W_J$. 
Write $\min(\tau W_J,\preceq_{\sigma w_J}) \in \tau W_J$ as $\tau w_J'$, with $w_J' \in W_{J}$.
\begin{enumerate}
\item There is a unique directed path in $\QB(W)$ from $\sigma w_J$ to some $x\in\tau W_J$ whose edge labels 
are increasing with respect to $<_{\lambda}$ and lie in $\Phi^+\setminus\Phi_J^+$. This path ends at $\tau w_J'$.
\item Assume that there is a directed path from $\sigma$ to $\tau$ in $\QB_{b\lambda}(W^J)$ for some $b\in{\mathbb Q}$. 
Then the path in~{\rm (1)} from $\sigma w_J$ to $\tau w_J'$ is in $\QB_{b\lambda}(W)$.
\end{enumerate}
\end{lem}

In order to state the analogue of Lemma~\ref{L:mainb} for the quantum left Deodhar lifts, namely Lemma~\ref{L:mainb1}, we need the reverse of the reflection order $<_\lambda$, which is denoted $<_\lambda^*$ (this has all the roots in $\Phi_J^+$ at the beginning). It is well-known that $<_\lambda^*$ is a reflection order as well. We also need the following two lemmas, which are proved in the same way as their counterparts in \cite{lnsumk1}, namely Lemmas~7.4~and~7.5 in this paper. 

\begin{lem}\label{lem1}
Assume that $\ell(x\Rightarrow v)$, as a function of $x\in w W_J$, has a minimum at $x=x_0$. Then the path from $x_0$ to $v$ 
with increasing edge labels with respect to $<_\lambda^*$ (cf. Theorem~{\rm \ref{thm:shell}~(1)}) has all its labels in $\Phi^+\setminus\Phi_J^+$.
\end{lem}

\begin{lem}\label{lem2}
Assume that the paths with increasing edge labels from two elements $x_0,x_1$ in $wW_J$ to $v$ (cf. Theorem~{\rm \ref{thm:shell}~(1)}) 
have all labels in $\Phi^+\setminus\Phi_J^+$. Then $x_0=x_1$.
\end{lem}

\begin{lem}\label{L:mainb1} 
Consider $\sigma,\tau\in W^J$ and $w_J\in W_J$. 
Write $\max(\sigma W_J,\preceq_{\tau w_J}^*) \in \sigma W_J$ as $\sigma w_J'$, with $w_J' \in W_{J}$.
\begin{enumerate}
\item There is a unique directed path in $\QB(W)$ from some $x\in\sigma W_J$ to $\tau w_J$ whose edge labels 
are increasing with respect to $<_\lambda^*$ and lie in $\Phi^+\setminus\Phi_J^+$. This path starts at $\sigma w_J'$.
\item Assume that there is a directed path from $\sigma$ to $\tau$ in $\QB_{b\lambda}(W^J)$ for some $b\in{\mathbb Q}$. 
Then the path in~{\rm (1)} from $\sigma w_J'$ to $\tau w_J$ is in $\QB_{b\lambda}(W)$.
\end{enumerate}
\end{lem}

\begin{proof} The proof is completely similar to that of Lemma~\ref{L:mainb}, i.e., \cite[Lemma~6.6]{lnsumk2}, based on Lemmas~\ref{lem1}, \ref{lem2}, \ref{bqbpaths}, and Theorem~\ref{thm:shell}~(2).
\end{proof}

\section{Chevalley formulas for semi-infinite flag manifolds}\label{csi}

Consider a connected, simply-connected simple algebraic group $G$ over $\mathbb{C}$, with Borel subgroup $B = T N$, maximal torus $T$, and unipotent radical $N$. The {\em semi-infinite flag manifold} $\mathbf{Q}_{G}^{\mathrm{rat}}$ associated to $G$ is an ind-scheme of infinite type whose set of $\mathbb{C}$-valued points is $G(\bC(\!(z)\!)\,)/\left(T(\bC)\cdot N(\bC(\!(z)\!))\right)$; 
note that $\mathbf{Q}_{G}^{\mathrm{rat}}$ is an inductive limit of copies of the (reduced) closed subscheme $\QQ$ of infinite type, introduced in \cite[Section~4.1]{FM} 
(for details, see \cite{kat2} and also \cite{kat4}). 
In this paper, we concentrate on the semi-infinite Schubert (sub)variety $\QQ = \QQ(e) \subset \mathbf{Q}_{G}^{\mathrm{rat}}$ corresponding to the identity element $e \in \Waf$, which we also call the semi-infinite flag manifold. 
Also, for each $x \in \Wafp = W \times Q^{\vee,+}$, one has the corresponding semi-infinite Schubert (sub)variety $\QQ(x) \subset \QQ$, which is the closure of the orbit under the Iwahori subgroup $\mathbf{I} \subset G(\mathbb{C}[\![z]\!])$ through the ($T \times \mathbb{C}^{*}$)-fixed point labeled by $x$ (in exactly the same way as in \cite[Section~4.2]{knsekt} and \cite[Section~2.3]{O}). 
The ($T \times \mathbb{C}^*$)-equivariant $K$-group $K_{T \times \mathbb{C}^*}(\QQ)$ of $\QQ$ 
{has a topological (in the sense of \cite[Proposition~5.11]{knsekt}) $\mathbb{Z}[q, q^{-1}][P]$-basis} of {\em semi-infinite Schubert classes}, and its multiplicative structure is determined by a {\em Chevalley formula}, which expresses the tensor product of a semi-infinite Schubert class with the class of a line bundle. In \cite{knsekt} and \cite{nospcf}, the Chevalley formulas were given in the case of a dominant and an anti-dominant weight $\lambda$, respectively. These formulas were expressed in terms of the quantum LS path model. We will express them in terms of the quantum alcove model based on the lexicographic $\lambda$-chain. The goal is to generalize these formulas for an arbitrary weight $\lambda$, and we will also see that an arbitrary $\lambda$-chain can be used. Throughout this section, $W_J$ is the stabilizer of $\lambda$, and we use freely the notation of Section~\ref{backg}.

More precisely, \revise{the ($T \times \mathbb{C}^*$)-equivariant $K$-group $K_{T \times \mathbb{C}^*}(\QQ)$ is the $\mathbb{Z}[q, q^{-1}][P]$-submodule of the Laurent series (in $q^{-1}$) extension $\mathbb{Z}(\!(q^{-1})\!)[P] \otimes_{\mathbb{Z}[\![q^{-1}]\!][P]} K_{\mathbf{I} \rtimes \mathbb{C}^{*}}^{\prime}(\mathbf{Q}_{G})$ of the equivariant (with respect to the Iwahori subgroup $\mathbf{I}$, together with the loop rotation action of $\mathbb{C}^{*}$) $K$-group $K_{\mathbf{I} \rtimes \mathbb{C}^{*}}^{\prime}(\mathbf{Q}_{G})$ of $\mathbf{Q}_{G}$, introduced in \cite{knsekt}, consisting of all infinite linear combinations of the classes $\OQG{x}$, $x \in \Wafp = W \times Q^{\vee,+}$, of the structure sheaf of the semi-infinite Schubert variety $\QQ(x) (\subset \QQ)$ with coefficient $a_{x} \in \mathbb{Z}[q, q^{-1}][P]$ such that the sum $\sum_{x \in \Wafp} \vert a_{x} \vert$ of the absolute values $\vert a_{x} \vert$ lies in $\mathbb{Z}_{\geq 0}[P](\!(q^{-1})\!)$; see \cite[Section~5]{knsekt} for details}. 
Here $\mathbb{Z}[P]$ is the group algebra of $P$, spanned by formal exponentials $\mathbf{e}^{\mu}$ for $\mu \in P$, with $\mathbf{e}^{\mu} \mathbf{e}^{\nu} =\mathbf{e}^{\mu+\nu}$, and it is identified with the representation ring of $T$. 
Note that for each $x \in \Wafp$ and $\nu \in P$, the twisted semi-infinite Schubert class $\OQGl{\nu} \cdot \OQG{x}$, defined by the tensor product in $K_{\mathbf{I} \rtimes \mathbb{C}^{*}}^{\prime}(\mathbf{Q}_{G})$, indeed lies in $K_{T \times \mathbb{C}^{*}}(\QQ)$; this is seen by using \cite[Theorem~5.16]{KLN1} and (the proof of) \cite[Corollary~5.12]{knsekt}. 
We also consider the $\mathbb{Z}[q, q^{-1}][P]$-submodule $K_{T \times \mathbb{C}^*}^{\prime}(\QQ)$ of $K_{T \times \mathbb{C}^*}(\QQ)$ consisting of all finite linear combinations of the classes $\OQG{x}$, $x \in \Wafp$, with coefficients in $\mathbb{Z}[q, q^{-1}][P]$.

The $T$-equivariant $K$-groups of $\QQ$, denoted by $K_{T}(\QQ)$ and $K_{T}^{\prime}(\QQ)$, are obtained from the $K_{T \times \mathbb{C}^*}(\QQ)$ and $K_{T \times \mathbb{C}^*}^{\prime}(\QQ)$, respectively, by the specialization $q = 1$. Hence the Chevalley formulas for $K_T(\QQ)$ (for arbitrary weights) and $K_{T}^{\prime}(\QQ)$ (for anti-dominant weights) are obtained from the corresponding ones for $K_{T \times \mathbb{C}^*}(\QQ)$ by setting $q=1$. 
More precisely, the $T$-equivariant $K$-group $K_{T}(\QQ)$ is defined to be the $\mathbb{Z}[P]$-module $\prod_{x \in \Wafp} \mathbb{Z}[P] \OQG{x}$ (direct product) consisting of all infinite linear combinations of the classes $\OQG{x}$, $x \in \Wafp$, with coefficients in $\mathbb{Z}[P]$; 
note that for each $\nu \in P$, a $\mathbb{Z}[P]$-linear endomorphism $\OQGl{\nu} \cdot \bullet$ of $K_{T}(\QQ)$ is induced from the $\mathbb{Z}[q, q^{-1}][P]$-linear endomorphism $\OQG{\nu} \cdot \bullet$ of $K_{T \times \mathbb{C}^{*}}(\QQ)$ by the specialization (of coefficients) at $q = 1$. 
\revise{Also, $K_{T}^{\prime}(\QQ)$ is defined to be the $\mathbb{Z}[P]$-submodule of $K_{T}(\QQ)$ consisting of all finite linear combinations of the classes $\OQG{x}$, $x \in \Wafp$, with coefficients in $\mathbb{Z}[P]$.}

\subsection{Chevalley formula for dominant weights} 
We start with the Chevalley formula for dominant weights, which was derived in terms of semi-infinite LS paths in~\cite{knsekt}, and then restated in \cite[Corollary~C.3]{nospcf} in terms of quantum LS paths. 

Let $\lambda = \sum_{i \in I} \lambda_i \varpi_i$ be a dominant weight. We denote by $\overline{{\rm Par}(\lambda)}$ the set of $I$-tuples of partitions ${\bm{\chi}} = (\chi^{(i)})_{i \in I}$ such that $\chi^{(i)}$ is a partition of length at most $\lambda_i$ for all $i \in I$. For ${\bm{\chi}} = (\chi^{(i)})_{i \in I} \in \overline{{\rm Par}(\lambda)}$, we set $|{\bm{\chi}}| := \sum_{i \in I} |\chi^{(i)}|$, with $|\chi^{(i)}|$ the size of the partition $\chi^{(i)}$. Also, set $\iota({\bm{\chi}}) := \sum_{i \in I} \chi^{(i)}_1 \alpha_i^{\vee} \in Q^{\vee,+}$, with $\chi^{(i)}_1$ the first part of the partition $\chi^{(i)}$.

\begin{thm}[\cite{knsekt,nospcf}]\label{chevdomqls} Let $x=wt_{\xi}\in \Wafp = W \times Q^{\vee,+}$. Then, in $K_{T \times \mathbb{C}^*}(\QQ)$, we have
\[\begin{split}&\OQGl{-w_\circ\lambda}\cdot\OQG{x}=\\[3mm]&\qquad=\sum_{\eta\in\QLS(\lambda)}\sum_{{\bm{\chi}}\in\overline{{\rm Par}(\lambda)}}q^{\Deg_w(\eta)-\langle\lambda,\xi\rangle-|{\bm{\chi}}|}\mathbf{e}^{\wt(\eta)}\OQG{\iota(\eta,w)t_{\xi+\xi(\eta,w)+\iota({\bm{\chi}})}}\,.\end{split}\]
\end{thm}

\begin{rema}\label{convqls} {\rm The original Chevalley formula for a dominant weight, as stated in \cite[Corollary~C.3]{nospcf}, is in terms of a slightly different version of quantum LS paths. They can be recovered from those in Definition~\ref{defls} simply by replacing the numbers $b_i$ with $1-b_i$ (arranged increasingly) and by reversing the second sequence in~\eqref{E:lsste}; indeed $\QB_{b\lambda}(W)$ is identical to $\QB_{(1-b)\lambda}(W)$. 
The same observation applies to the original Chevalley formula for an anti-dominant weight, as stated in \cite[Theorem~1]{nospcf}; see Theorem~\ref{chevantidomqls} below. }
\end{rema}

We now translate this formula in terms of the quantum alcove model for the lex $\lambda$-chain $\Gamma_{\rm lex}(\lambda)$. To this end, given $w\in W$, we construct a bijection $A\mapsto \eta$ between ${\mathcal A}(w,\Gamma_{\rm lex}(\lambda))$ and $\QLS(\lambda)$, for which several properties are then proved. 

In order to construct the forward map, let $A=\{j_1<\cdots<j_s\}$ be in ${\mathcal A}(w,\Gamma_{\rm lex}(\lambda))$. The corresponding heights are within the first range in~\eqref{ranges}. Consider the weakly increasing sequence of relative heights
\begin{equation}\label{relh}h_i:=\frac{l_{j_i}}{\langle\lambda,\beta_{j_i}^\vee\rangle}\,\in[0,1)\cap{\mathbb{Q}}\,,\;\;\;\;\;i=1,\ldots,s\,.\end{equation}
Let $0<b_2<\cdots<b_t<1$ be the distinct nonzero values in the set $\{h_1,\ldots,h_s\}$, and let $b_1:=0$, $b_{t+1}:=1$. For $k=1,\ldots,t$, let $I_k:=\{1 \leq i \leq s \, \mid \, h_i=b_k\}$; these sets are all non-empty, except perhaps $I_1$. 

Recall the path $\Pi(w,A)$ in $\QB(W)$ given by~\eqref{eqn:admissible}. We divide this path into subpaths corresponding to the sets $I_k$, and record the last element in each subpath; more precisely, for $k=0,\ldots,t$, we define the sequence of Weyl group elements
\[w_k:=w\prod_{i\in I_1\cup\cdots\cup I_k}^{\longrightarrow} r_{j_i}\,,\]
where the non-commutative product is taken in the increasing order of the indices $i$; in particular, $w_0:=w$. For $k=1,\ldots,t$, let $\sigma_k:=\lfloor w_k\rfloor\in W^J$. We can now define the forward map as
\[(w,A)\,\mapsto\,\eta:=((b_1,b_2,\ldots,b_t,b_{t+1});\,(\sigma_1,\ldots,\sigma_t))\,.\]
We will verify below that the image is in $\QLS(\lambda)$. 

The inverse map is constructed using the quantum right Deodhar lift and the related shellability property of the quantum Bruhat graph. We begin with a 
quantum LS path $\eta\in\QLS(\lambda)$ of the form \eqref{E:lsste}. Letting $w_0=w$, define the lifts
\begin{equation}\label{ups}
  w_k = \min(\sigma_kW_J,\preceq_{w_{k-1}})\qquad\text{for $k=1,\ldots,t\,$.}
\end{equation}
By Lemma \ref{L:mainb}, for each $k=1,\ldots,t$, there is a unique
directed path from $w_{k-1}$ to $w_k$ in $\QB_{b_k\lambda}(W)$ with labels in $\Phi^+\setminus\Phi_J^+$, 
which are increasing with respect to the reflection order $<_\lambda$. 
Let us replace each label $\beta$ in this path with the pair $(\beta,\,b_k\pair{\lambda}{\beta^\vee})$, 
where the second component is in $\{0,1,\ldots,\pair{\lambda}{\beta^\vee}-1 \}$, by the definition of $\QB_{b_k\lambda}(W)$.  
Thus, each such pair defines a $\lambda$-hyperplane. 
By concatenating these paths, we obtain a directed path in $\QB(W)$ starting at $w$, together with
a sequence of $\lambda$-hyperplanes. We will show that this sequence is lex-increasing, and thus it defines a $w$-admissible subset.

\begin{prop}\label{qam2qls} The map $A\mapsto \eta$ constructed above is a bijection between ${\mathcal A}(w,\Gamma_{\rm lex}(\lambda))$ and $\QLS(\lambda)$. It maps the corresponding parameters in the following way{\rm :}
\begin{equation}\label{mapstats}\wt(w,A)=\wt(\eta),\;\; {\rm end}(w,A)=\iota(\eta,w),\;\; {\rm down}(w,A)=\xi(\eta,w),\;\; -{\rm height}(w,A)=\Deg_w(\eta).\end{equation}
\end{prop}

\begin{proof} 
We start by showing that the forward map is well-defined. By the definition of relative height, the subpath of $\Pi(w,A)$ from $w_{k-1}$ to $w_k$ is in $\QB_{b_k\lambda}(W)$. Thus, Lemma~\ref{q2ls} implies that $\eta\in\QLS(\lambda)$. For the well-definedness of the
inverse map, it suffices to prove that the constructed sequence of $\lambda$-hyperplanes is lex-increasing. Indeed, the relative heights of the 
$\lambda$-hyperplanes are the numbers $b_k$, and hence they weakly increase by construction; on the other hand, 
within the same relative height, the $\lambda$-hyperplanes increase because 
of the compatibility of the reflection order $<_\lambda$ with the lex $\lambda$-chain (see Remark~\ref{remlex}).

To show that the two maps are mutually inverse, the crucial fact to check is that
the forward map composed with the backward one is the identity.
This follows from the uniqueness part in Lemma~\ref{L:mainb}~(1), after recalling again Remark~\ref{remlex}. 
In particular, the subsequence $w_k\in W$ of the original path $\Pi(w,A)$ in $\QB(W)$ is reconstructed by the inverse map via \eqref{ups}; furthermore, this construction is identical with the one in~\eqref{eq:tiw}, on which the definitions of $\iota(\eta,w)$, $\xi(\eta,w)$, and $\Deg_w(\eta)$ are based. Thus, the last three properties in~\eqref{mapstats} follow. To be more precise, for the last one we note that, if the relative height of the $\lambda$-hyperplane $H_{\beta_j,-l_j}$ is $b_k$ (for $j\in A$, cf.~\eqref{relh}), then we have
\begin{equation}\label{contribheight}(1-b_k)\pair{\lambda}{\beta_j^\vee}=\pair{\lambda}{\beta_j^\vee}-l_j=\widetilde{l}_j\,.\end{equation}

Finally, the weight preservation follows via the same argument as in the proof of~\cite[Proposition~4.18]{LS}, which extends to the present setup by~\cite[Remark~4.19]{LS}. Indeed, the above construction of the map from ${\mathcal A}(w,\Gamma_{\rm lex}(\lambda))$ to $\QLS(\lambda)$ is completely similar to that of the map in the mentioned proof. 
\end{proof}

We translate the formula in Theorem~\ref{chevdomqls} to the quantum alcove model via Proposition~\ref{qam2qls}.

\begin{thm}\label{chevdomqam} Let $\lambda$ be a dominant weight, $\Gamma_{\rm lex}(\lambda)$ the lex $\lambda$-chain, and let $x=wt_{\xi}\in \Wafp$. Then, in $K_{T \times \mathbb{C}^*}(\QQ)$, we have
\[\begin{split}&\OQGl{-w_\circ\lambda}\cdot\OQG{x}=\\[3mm]&\qquad \sum_{A\in{\mathcal A}(w,\Gamma_{\rm lex}(\lambda))}\sum_{{\bm{\chi}}\in\overline{{\rm Par}(\lambda)}}q^{-{\rm height}(w,A)-\langle\lambda,\xi\rangle-|{\bm{\chi}}|}\mathbf{e}^{{\rm wt}(w,A)}\OQG{{\rm end}(w,A)t_{\xi+{\rm down}(w,A)+\iota({\bm{\chi}})}}\,.\end{split}\]
\end{thm}

\subsection{Chevalley formula for anti-dominant weights} 

We continue with the Chevalley formula for an anti-dominant weight $\lambda$, which was derived in terms of quantum LS paths in~\cite[Theorem~1]{nospcf}. 

\begin{thm}[\cite{nospcf}]\label{chevantidomqls} Let $\lambda$ be an anti-dominant weight, and let $x=wt_{\xi}\in \Wafp$. Then, in $K_{T \times \mathbb{C}^*}^{\prime}(\QQ) \subset K_{T \times \mathbb{C}^*}(\QQ)$, we have 
\[\begin{split}&\OQGl{-w_\circ\lambda}\cdot\OQG{x}=\\[3mm]&\qquad=\sum_{v\in W}\sum_{\substack{\eta\in\QLS(-\lambda) \\ \kappa(\eta,v)=w}}(-1)^{\ell(v)-\ell(w)}q^{-\Deg(\eta)-\langle\lambda,\xi\rangle}\mathbf{e}^{-\wt(\eta)}\OQG{vt_{\xi+\zeta(\eta,v)}}\,.\end{split}\]
\end{thm}

We now translate this formula in terms of the quantum alcove model for the lex $\lambda$-chain $\Gamma_{\rm lex}(\lambda)$, which is defined just as the reverse of the lex $(-\lambda)$-chain described in Proposition~\ref{speciallch}; note that the alcove path corresponding to the former (ending at $A_\circ-\lambda$) is just the translation by $-\lambda$ of the alcove path corresponding to the latter (ending at $A_\circ+\lambda$). Given $w\in W$, we construct a bijection $A\mapsto (\eta,v)$ between ${\mathcal A}(w,\Gamma_{\rm lex}(\lambda))$ and the set
\[\bm{\QLS}_w(-\lambda):=\{(\eta,v)\, \mid \,\eta\in\QLS(-\lambda),\,v\in W,\,\kappa(\eta,v)=w\} \,.\]

The construction of the bijection is very similar to the one above, in the dominant case, and so we only highlight the differences. In order to construct the forward map, let $A=\{j_1<\cdots<j_s\}$ be in ${\mathcal A}(w,\Gamma_{\rm lex}(\lambda))$. The corresponding heights are within the second range in~\eqref{ranges}, while the relative heights $h_i$, defined as in~\eqref{relh}, belong to $(0,1]\cap{\mathbb Q}$. The numbers $b_k$ are defined in the same way, for $k=1,\ldots,t+1$, as are the sets $I_k$, for $k=2,\ldots,t+1$; all of the latter are non-empty, except perhaps $I_{t+1}$. Then, for $k=1,\ldots,t+1$, we define
\[w_k:=w\prod_{i\in I_2\cup\cdots\cup I_k}^{\longrightarrow} r_{j_i}\;\;\;\;\;\mbox{(in particular, $w_1:=w$)}\,,\]
and the forward map as
\[(w,A)\,\mapsto\,(\eta:=((b_1,b_2,\ldots,b_t,b_{t+1});\,(\sigma_1,\ldots,\sigma_t)),\;w_{t+1})\,,\;\;\;\;\;\mbox{where $\sigma_k:=\lfloor w_k\rfloor$}\,.\]

For the inverse map, we start with $(\eta,v)\in\bm{\QLS}_w(-\lambda)$, and construct the sequence $w_k$, for $k=1,\ldots,t+1$, via the quantum left Deodhar lifts, as in~\eqref{eq:tiw1}. By Lemma \ref{L:mainb1}, for each $k=1,\ldots,t$ there is a unique
directed path from $w_{k}$ to $w_{k+1}$ in $\QB_{b_k\lambda}(W)$ with labels $|\beta|$ for $\beta\in\Phi^-\setminus\Phi_J^-$, 
which are increasing with respect to the reflection order $<_\lambda^*$. Like in the dominant case, by concatenating these paths we obtain a directed path in $\QB(W)$ starting at $w$, together with
a sequence of $\lambda$-hyperplanes $(\beta,\,b_k\pair{\lambda}{\beta^\vee})$, 
where $\beta$ are the above labels, and the second component is in $\{1,\ldots,\pair{\lambda}{\beta^\vee} \}$.

\begin{prop}\label{qam2qls1}
The map $A\mapsto (\eta,v)$ constructed above is a bijection between ${\mathcal A}(w,\Gamma_{\rm lex}(\lambda))$ and $\bm{\QLS}_w(-\lambda)$. It maps the corresponding parameters in the following way{\rm :} 
\begin{equation}\label{mapstats1}\wt(w,A)=-\wt(\eta),\;\; {\rm end}(w,A)=v,\;\; {\rm down}(w,A)=\zeta(\eta,v),\;\; {\rm height}(w,A)=\Deg(\eta).\end{equation}
\end{prop}

\begin{proof} This proof is completely similar to that of Proposition~\ref{qam2qls}, and so we highlight the minor differences. To show that the two maps are mutually inverse, we use the uniqueness part in Lemma~\ref{L:mainb1}~(1). 

Another difference is concerned with proving ${\rm height}(w,A)=\Deg(\eta)$. Note first that 
\[{\rm height}(w,A)={\rm height}(w,\overline{A})\,,\]
 where $\overline{A}$ is the subset of $A$ which corresponds to ignoring the $\lambda$-hyperplanes of relative height equal to $1$; indeed, the contribution of each such hyperplane is $0$, see~\eqref{defheight}. Thus, ${\rm height}(w,A)$ is defined based on shortest directed paths in $\QB(W)$ from $w_{k-1}$ to $w_{k}$, for $k=2,\ldots,t$. Comparing with the definition~\eqref{eq:deg} of $\Deg(\eta)$, where $\sigma_k:=\lfloor w_k\rfloor$, and using Lemma~\ref{relwt}, as well as the analogue of~\eqref{contribheight}, the desired equality is proved.
\end{proof}

We translate the formula in Theorem~\ref{chevantidomqls} to the quantum alcove model via Proposition~\ref{qam2qls1}. We use the notation $|A|$ to indicate the cardinality of the set $A$.

\begin{thm}\label{chevantidomqam} Let $\lambda$ be an anti-dominant weight, $\Gamma_{\rm lex}(\lambda)$ the lex $\lambda$-chain, and let $x=wt_{\xi}\in \Wafp$. Then, in $K_{T \times \mathbb{C}^*}^{\prime}(\QQ) \subset K_{T \times \mathbb{C}^*}(\QQ)$, we have
\[\begin{split}&\OQGl{-w_\circ\lambda}\cdot\OQG{x}=\\[3mm]&\qquad\sum_{A\in{\mathcal A}(w,\Gamma_{\rm lex}(\lambda))}(-1)^{|A|}q^{-{\rm height}(w,A)-\langle\lambda,\xi\rangle}\mathbf{e}^{{\rm wt}(w,A)}\OQG{{\rm end}(w,A)t_{\xi+{\rm down}(w,A)}}\,.
\end{split}\]
\end{thm}

\subsection{Chevalley formula for arbitrary weights} 
We now state the Chevalley formula for an arbitrary weight $\lambda=\sum_{i\in I}\lambda_i\varpi_i$; 
this is the natural common generalization of Theorems~\ref{chevdomqam} and \ref{chevantidomqam}. 
In order to exhibit the general formula, let $\overline{{\rm Par}(\lambda)}$ denote 
the set of $I$-tuples of partitions $\bchi=(\chi^{(i)})_{i\in I}$ such that $\chi^{(i)}$ is 
a partition of length at most $\max(\lambda_i,0)$.

\begin{thm}\label{genchev} Let $\lambda$ be an arbitrary weight, 
$\Gamma(\lambda)$ an arbitrary reduced $\lambda$-chain, and let $x=wt_{\xi}\in \Wafp = W \times Q^{\vee,+}$. 
Then, in $K_{T\times{\mathbb C}^*}(\QQ)$, we have
\[\begin{split}&\OQGl{-w_\circ\lambda}\cdot\OQG{x}=\\[3mm]&\quad\!\!\!\!\sum_{A\in{\mathcal A}(w,\Gamma(\lambda))}\sum_{{\bm{\chi}}\in\overline{{\rm Par}(\lambda)}}(-1)^{n(A)}q^{-{\rm height}(w,A)-\langle\lambda,\xi\rangle-|{\bm{\chi}}|}\mathbf{e}^{{\rm wt}(w,A)}\OQG{{\rm end}(w,A)t_{\xi+{\rm down}(w,A)+\iota({\bm{\chi}})}}\,,
\end{split}\]
where $n(A)$, for $A=\{j_1<\cdots<j_s\}$, is the number of negative roots in $\{\beta_{j_1},\ldots,\beta_{j_s}\}$.
\end{thm}

\begin{exa}{\rm 
Assume that $\mathfrak{g}$ is of type $A_2$, and $\lambda=\varpi_1-\varpi_2$. 
Then, $\Gamma(\lambda):=(\alpha_{1},\,-\alpha_{2})$ is a reduced $\lambda$-chain.
Assume that $w = s_{1}=s_{\alpha_{1}}$. In this case, we see that
${\mathcal A}(s_{1},\Gamma(\lambda))=\bigl\{\emptyset,\,\{1\},\,\{2\},\,\{1,2\}\bigr\}$, 
and we have the following table: 
\begin{equation*}
\begin{array}{c||c|c|c|c|c}
A & n(A) & {\rm height}(s_{1},A) & {\rm wt}(s_{1},A) & {\rm end}(s_{1},A) & {\rm down}(s_{1},A) \\ \hline\hline
\emptyset & 0 & 0 & s_{1}\lambda & s_{1} & 0 \\ \hline
\{1\} & 0 & 1 & \lambda & e & \alpha_{1}^{\vee} \\ \hline
\{2\} & 1 & 0 & s_{1}\lambda & s_{1}s_{2} & 0 \\ \hline
\{1,2\} & 1 & 1 & \lambda & s_{2} & \alpha_{1}^{\vee}
\end{array}
\end{equation*}
Also, we identify $\overline{{\rm Par}(\lambda)}$ with $\mathbb{Z}_{\ge 0}$. 
Therefore, we obtain 
\begin{equation*}
\begin{split}
 & \OQGl{-w_\circ\lambda} \cdot \OQG{s_{1}t_{\xi}} = \\[3mm]
 & \hspace*{5mm} 
  \sum_{ m \in \mathbb{Z}_{\ge 0} } q^{ -\langle \lambda,\xi \rangle-m}
  \biggl\{%
  \underbrace{
  \mathbf{e}^{s_1\lambda}
  \OQG{ s_1 t_{\xi+m\alpha_{1}^{\vee}} }}_{A=\emptyset} 
  + 
  \underbrace{
  q^{-1} \mathbf{e}^{\lambda}
  \OQG{ t_{\xi+\alpha_{1}^{\vee}+m\alpha_{1}^{\vee}} }}_{A=\{1\}} \\
 & \hspace*{40mm}
   + 
   \underbrace{(-1)\mathbf{e}^{s_{1}\lambda}
   \OQG{ s_1s_2 t_{\xi+m\alpha_{1}^{\vee}} }}_{A=\{2\}}
   +
   \underbrace{ (-1)q^{-1} \mathbf{e}^{\lambda}
   \OQG{ s_2 t_{\xi+\alpha_{1}^{\vee}+m\alpha_{1}^{\vee}} }}_{A=\{1,2\}} \biggr\}\,. 
\end{split}
\end{equation*}}
\end{exa}

As an immediate consequence of Theorem~\ref{genchev}, we obtain the semi-infinite analog of the duality formulas \cite[Theorems~8.6 and 8.7]{LP}, which hold in $K_{T}(\QQ)$ (not in $K_{T \times \mathbb{C}^*}(\QQ)$).
For $\zeta \in Q^{\vee,+}$, we define the following $\mathbb{Z}[P]$-linear operator (acting on the right) on $K_{T}(\QQ)$:
\[\OQG{x}\cdot t_{\zeta}:=\OQG{xt_\zeta}, \quad x \in \Waf;\]
we also consider an arbitrary (possibly, infinite) sum, with coefficients in $\mathbb{Z}[P]$, of 
the operators $t_{\zeta}$, $\zeta \in Q^{\vee,+}$, which is a well-defined operator on $K_{T}(\QQ)$.
Now, for an arbitrary $\lambda \in P$, we introduce the following operator on $K_{T}(\QQ)$:
\[c_{w}^{v}(\lambda):=\sum_{\substack{A\in\CA(w,\Gamma(\lambda)) \\ \ed(w,A)=v}}(-1)^{n(A)} \be^{\wt(w,A)} t_{\dn(w,A)}
  \quad \text{for $v,w \in W$}.\]
Then, we can express the general Chevalley formula for $q=1$, that is, the general Chevalley formula for $K_{T}(\QQ)$, as:
\[\OQGl{-w_\circ\lambda}\cdot\OQG{x}=\sum_{v\in W}\OQG{v}
\left(c_{w}^v(\lambda)\sum_{{\bm{\chi}}\in\overline{{\rm Par}(\lambda)}}t_{\xi+\iota(\bchi)}\right)\]
for an arbitrary $\lambda \in P$ and $x = wt_{\xi} \in \Wafp$.

\begin{cor} \label{cor:duality}
Let $\lambda \in P$. For $v,w \in W$, we have the following equalities for the operators $c_{w}^{v}(\lambda)${\rm :} 
\begin{align}
&c_w^v(\lambda)=
(-1)^{\ell(w,v)}\theta(c^{ww_\circ}_{vw_\circ}(w_\circ\lambda)), \label{eq:d1} \\
&c_w^v(\lambda)=
(-1)^{\ell(w,v)}\eta(c^{w_\circ w}_{w_\circ v}(-\lambda)), \label{eq:d2}
\end{align}
where $\ell(w,v)$ denotes the length of a shortest directed path from $w$ to $v$ in $\mathrm{QB}(W)$, 
while $\eta: \be^{\mu} \mapsto \be^{-w_\circ\mu}$ and $\theta: t_{\zeta} \mapsto t_{-w_\circ\zeta}$ for $\mu \in P$ and $\zeta \in Q^{\vee,+}$. 
\end{cor}

\begin{proof}
Equalities \eqref{eq:d1} and \eqref{eq:d2} can be proved by arguments similar to those in the proofs of 
\cite[Theorems~8.6 and 8.7]{LP}, respectively;
in addition, we make use of the following facts about $\mathrm{QB}(W)$:
\begin{itemize}
\item the maps $w\mapsto ww_\circ$ and $w\mapsto w_\circ w$ are anti-automorphisms of $\mathrm{QB}(W)$;
\item the lengths of all paths from $w$ to $v$ in $\mathrm{QB}(W)$ have the same parity.
\end{itemize}
For the latter fact, we refer to \cite[Section 6]{BFP}. More precisely, by using the argument in the last paragraph of the proof of \cite[Theorem~6.4]{BFP}, including the related setup, we can show that any path from $w$ to $v$ in $\mathrm{QB}(W)$ can be transformed into the unique label-increasing one by using  \cite[Lemma~6.7]{BFP}, combined with removing loops of length 2. The needed fact immediately follows.
\end{proof}

\begin{rema}
{\rm By combining equations \eqref{eq:d1} and \eqref{eq:d2}, 
we obtain 
\[c_{w}^{v}(-w_{\circ} \lambda) = \theta \eta(c_{w_{\circ} w w_{\circ}}^{w_{\circ} v w_{\circ}}(\lambda)),\]
which is the semi-infinite analog of \cite[Corollary~8.8]{LP}.
This equality can also be explained (as in the geometric proof of \cite[Proposition~8.9]{LP}) by using the Dynkin diagram automorphism induced by ``$-w_{\circ}$''; 
see \cite[Remark A.4]{KNOS}.}
\end{rema}

\section{Proof of Theorem~{\rm \ref{genchev}}} \label{secproof}

\subsection{Quantum Bruhat operators at ``$q=1$''} 
Let $\vartheta^{\vee}$ be the highest coroot for the set $\Phi^{\vee}$ $:= \{ \alpha^{\vee} \mid \alpha \in \Phi \}$ of roots of $\mathfrak{g}^{\vee}$; the element $\vartheta^{\vee}$ should not be confused with the coroot $\theta^{\vee}$ of the highest root $\theta$ for the set $\Phi$ of roots of $\mathfrak{g}$. 
Let $h := \pair{\rho}{\vartheta^{\vee}} + 1$ denote the Coxeter number of $\mathfrak{g}^{\vee}$, and consider the group algebra $\mathbb{Z}[P/h] \supset \mathbb{Z}[P]$. We set 
$\tilde{K}_{T}(\QQ) := K_{T}(\QQ) \otimes_{\mathbb{Z}[P]} \mathbb{Z}[P/h]$; 
recall that the $T$-equivariant $K$-group $K_{T}(\QQ)$ consists of 
all (possibly infinite) linear combinations of the classes $\OQG{x}$, 
$x \in \Wafp = W \times Q^{\vee,+}$, with coefficients in $\mathbb{Z}[P]$. 
For a positive root $\beta \in \Phi^{+}$, we define a $\mathbb{Z}[P/h]$-linear operator $\SQ_{\beta}$ on $\tilde{K}_{T}(\QQ)$ by:
%
\begin{equation} \label{eq:qbo1}
\SQ_{\beta}\OQG{w t_{\xi}} : = 
\begin{cases}
  \OQG{ ws_{\beta} t_{\xi} }
  & \text{if $w \xrightarrow{\beta} ws_{\beta}$ is a Bruhat edge in $\QB(W)$}\,, \\[1mm]
  \OQG{ ws_{\beta} t_{\xi+\beta^{\vee}} }
  & \text{if $w \xrightarrow{\beta} ws_{\beta}$ is a quantum edge in $\QB(W)$}\,, \\[1mm]
  0 & \text{otherwise}\,,
\end{cases}
\end{equation}
where $w \in W$ and $\xi \in Q^{\vee,+}$. 
Also, we set $\SQ_{-\beta}:=-\SQ_{\beta}$ for $\beta \in \Phi^{+}$. 
For a weight $\nu \in P$, we define
\begin{equation}
\SX^{\nu} \OQG{ w t_{\xi} } = 
  \be^{w \nu / h} \OQG{ w t_{\xi} }\,, 
\end{equation}
where $w \in W$ and $\xi \in Q^{\vee,+}$. 
For $i \in I$, we define
\begin{equation}
\st_{i} \OQG{ x  } = \OQG{x t_{ \alpha_{i}^{\vee}} } \qquad \text{for $x \in \Wafp$}\,. 
\end{equation}
The following lemma is easily shown; cf. \cite[Equations (10.3)--(10.5)]{LP}. 
%
%
\begin{lem} \mbox{} \label{lem:QXt1}
\begin{enumerate}
\item We have $\SQ_{\pm \beta}^{2}=0$ for $\beta \in \Phi^{+} \setminus \Pi$, 
where $\Pi = \{ \alpha_i \}_{i \in I}$ is the set of simple roots. 
For $i \in I$, we have $\SQ_{\pm \alpha_{i}}^{2} = \st_{i}$, 
$\SQ_{\alpha_{i}}\SQ_{-\alpha_{i}} = \SQ_{-\alpha_{i}}\SQ_{\alpha_{i}} = - \st_{i}$, and 
\begin{equation*}
(\SX^{\alpha_{i}}+\SQ_{\alpha_{i}})
(\SX^{-\alpha_{i}}+\SQ_{- \alpha_{i}}) = 
(\SX^{-\alpha_{i}}+\SQ_{- \alpha_{i}})
(\SX^{\alpha_{i}}+\SQ_{\alpha_{i}})=1-\st_{i}\,. 
\end{equation*}

\item We have $\SX^{\nu}\SX^{\mu}=\SX^{\mu+\nu}$ for $\mu,\nu \in P$. 

\item We have $\SQ_{\beta} \SX^{\nu} = \SX^{s_{\beta}\nu} \SQ_{\beta}$ 
for $\nu \in P$ and $\beta \in \Phi$. 
\end{enumerate}
\end{lem}

We set
\begin{equation} \label{eq:sr1}
\SR_{\beta}:=\SX^{\rho} (\SX^{\beta}+\SQ_{\beta}) \SX^{-\rho} \qquad 
\text{for $\beta \in \Phi$}\,.
\end{equation}
%
%
\begin{prop} \label{prop:YBq1}
The family $\bigl\{ \SR_{\beta} \mid \beta \in \Phi \bigr\}$ satisfies 
the {\em Yang-Baxter equation}.  Namely, if $\alpha,\beta \in \Phi$ satisfy 
$\pair{\alpha}{\beta^{\vee}} \le 0$, or equivalently, 
$\pair{\beta}{\alpha^{\vee}} \le 0$, then 
\begin{equation} \label{eq:YB}
\SR_{\alpha}\SR_{s_{\alpha} \beta}
\SR_{s_{\alpha}s_{\beta} \alpha} \cdots 
\SR_{s_{\beta} \alpha}\SR_{\beta} = 
\SR_{\beta}\SR_{s_{\beta} \alpha} \cdots 
\SR_{s_{\alpha}s_{\beta} \alpha}
\SR_{s_{\alpha} \beta}\SR_{\alpha}\,. 
\end{equation}
\end{prop}

\begin{proof}
We set $\widetilde{\SR}_{\beta}:=1+\SQ_{\beta}$ for $\beta \in \Phi$. 
It follows from \cite[Corollary~4.4]{BFP} that 
the family $\bigl\{ \widetilde{\SR}_{\beta} \mid \beta \in \Phi^{+} \bigr\}$ 
satisfies the Yang-Baxter equation;
to apply this corollary, in view of the $\mathbb{Z}[P]$-module isomorphism 
from $QK_{T}(G/B) = K_{T}(G/B) \otimes_{\mathbb{Z}[P]} \mathbb{Z}[P][\![Q]\!]$ onto $K_{T}(\QQ)$, explained in Section~\ref{qkgb} below,
we take a field $k$ containing the ring 
$\mathbb{Z}[\![Q^{\vee,+}]\!] = \mathbb{Z}[\![Q_{i} \mid i \in I]\!]$ of 
formal power series in the variables $Q_{i} = Q^{\alpha_{i}^{\vee}}$, $i \in I$, 
and a $k$-valued multiplicative function $E$ on $\Phi^{+}$ given by $E(\alpha_{i}) := Q_{i}$ for each $i \in I$.

In order to prove that 
the family $\bigl\{ \widetilde{\SR}_{\beta} \mid \beta \in \Phi \bigr\}$ also 
satisfies the Yang-Baxter equation, we make use of the following observation. 
Noting that the leftmost operator (say $\widetilde{\SR}_{\alpha}$) 
on the left-hand side of the Yang-Baxter equation is identical to 
the rightmost operator on the right-hand side of the equation,
we multiply both sides of the Yang-Baxter equation by the operator $\widetilde{\SR}_{-\alpha}$ on the left and on the right. 
If $\alpha$ is not a simple root (resp., $\alpha=\alpha_{i}$ for some $i \in I$), 
then the leftmost two operators $\widetilde{\SR}_{-\alpha}\widetilde{\SR}_{\alpha}$ on 
the left-hand side and the rightmost two operators
$\widetilde{\SR}_{\alpha}\widetilde{\SR}_{-\alpha}$ on the right-hand side 
are both identical to $1$ (resp., $1 - \st_{i}$) 
by Lemma~\ref{lem:QXt1}~(1). 
Here we remark that the operator $1 - \st_{i}$ on 
$K_{T}(\mathbf{Q}_{G})$ is invertible, with its inverse $(1 - \st_{i})^{-1} = 
1+\st_{i}+\st_{i}^{2} + \cdots$, and commutes with 
$\widetilde{\SR}_{\gamma}$ for all $\gamma \in \Phi$. 
Hence, in the case that $\alpha=\alpha_{i}$, 
we can remove the operator $1 - \st_{i}$ from both sides of the equation 
by multiplying both sides by the inverse $(1 - \st_{i})^{-1}$. 
With this observation, the same argument as for 
\cite[Lemma 9.2]{LP} shows that 
the family $\bigl\{ \widetilde{\SR}_{\beta} \mid \beta \in \Phi \bigr\}$ also 
satisfies the Yang-Baxter equation. 

Now our assertion can be proved in exactly the same as \cite[Theorem 10.1]{LP}; 
use the commutation relations in Lemma~\ref{lem:QXt1} instead of 
\cite[Equations (10.3)--(10.5)]{LP} in the proof of \cite[Theorem 10.1]{LP}. 
\end{proof}

\begin{rema}
{\rm The Yang-Baxter property, as stated in Proposition~\ref{prop:YBq1}, is 
a weaker version of the similar property in~\cite[Definition~9.1]{LP}. 
Indeed, the additional requirement in the mentioned definition is that $R_{-\alpha}=(R_{\alpha})^{-1}$. 
By Lemma~\ref{lem:QXt1}, this still holds in our case if $\alpha$ is not a simple root, 
whereas $R_{-\alpha}=(1 - \st_{i})(R_\alpha)^{-1}$ when $\alpha=\alpha_{i}$ for some $i \in I$.}
\end{rema}

Let $\lambda \in P$ be an arbitrary weight. Recall that a reduced $\lambda$-chain 
$\Gamma=(\beta_{1},\dots,\beta_{m})$ corresponds to the following reduced alcove path:
\begin{equation} \label{eq:Gamma}
A_{\circ}=A_{0} 
  \xrightarrow{-\beta_{1}} A_{1} 
  \xrightarrow{-\beta_{2}} \cdots 
  \xrightarrow{-\beta_{m}} A_{m}=A_{-\lambda}\ (=A_{\circ}-\lambda)\,.
\end{equation}
%
%
%

%
\begin{rema} \label{rem:chain}
{\rm 
Let $\Gamma$ be a reduced $\lambda$-chain, and 
let $\Gamma'$ be an arbitrary (not necessarily reduced) $\lambda$-chain. 
We deduce from the proof of \cite[Lemma~9.3]{LP} that 
$\Gamma$ can be obtained from $\Gamma'$ by a sequence of 
the following two procedures {\rm (YB)} and {\rm (D)}: 
\begin{enumerate}
\item[(YB)] for $\alpha,\beta \in \Phi$ 
such that $\pair{\alpha}{\beta^{\vee}} \le 0$, or equivalently, 
$\pair{\beta}{\alpha^{\vee}} \le 0$, one replaces a segment 
of the form 
$\alpha,\,s_{\alpha} \beta,\,
s_{\alpha}s_{\beta} \alpha,\,\dots,\,
s_{\beta} \alpha,\,\beta$ by
$\beta,\,s_{\beta} \alpha,\,\dots,\,
s_{\alpha}s_{\beta} \alpha,\,
s_{\alpha} \beta,\,\alpha$\,; 

\item[(D)] one deletes a segment of the form $\beta$, $-\beta$ for $\beta \in \Phi$. 
\end{enumerate}}
\end{rema}

\subsection{Quantum Bruhat operators for generic $q$}

For simplicity of notation, 
we write $[\mathcal{O}_{\mathbf{Q}_{G}(x)}]$ as $[x]$ 
for $x \in W_{\af}^{\ge 0} = W \times Q^{\vee,+}$.
For $\beta \in \Phi$ and $k \in \mathbb{Z}$, we define 
a $\mathbb{Z}[q,q^{-1}][P/h]$-linear operator $\SQ_{\beta,k}$ on 
$\tilde{K}_{T \times \mathbb{C}^*}(\QQ) := 
K_{T \times \mathbb{C}^*}(\QQ) \otimes_{\mathbb{Z}[q,q^{-1}][P]} 
\mathbb{Z}[q,q^{-1}][P/h]$ as follows; recall that, by definition,
the ($T \times {\mathbb C}^*$)-equivariant 
$K$-group $K_{T \times {\mathbb C}^*}(\QQ)$ consists of 
all infinite linear combinations of the classes $[x]$, 
$x \in \Wafp$, with coefficients $a_{x} \in \mathbb{Z}[q,q^{-1}][P]$ such that 
the sum $\sum_{x \in \Wafp} \vert a_{x} \vert$ of the absolute values $\vert a_{x} \vert$ 
lies in $\mathbb{Z}_{\geq 0}[P](\!(q^{-1})\!)$:
\begin{equation}
\SQ_{\beta,k}[u t_{\xi}] = 
\begin{cases}
  \sgn(\beta) [ us_{\beta} t_{\xi} ]
  & \text{if $u \edge{|\beta|} us_{\beta}$ is a Bruhat edge in $\QB(W)$}, \\[2mm]
  \sgn(\beta) q^{-\sgn(\beta) k}[ us_{\beta} t_{\xi+|\beta|^{\vee}} ]
  & \text{if $u \edge{|\beta|} us_{\beta}$ is a quantum edge in $\QB(W)$}, \\[2mm]
  0 & \text{otherwise}, 
\end{cases}
\end{equation}
where $u \in W$ and $\xi \in Q^{\vee,+}$.
For a weight $\nu \in P$, we define
\begin{equation}
\SX^{\nu} [u t_{\xi}] = \be^{u \nu / h}[u t_{\xi}], 
\end{equation}
where $u \in W$ and $\xi \in Q^{\vee,+}$.
The following lemma is shown in the same way as Lemma~\ref{lem:QXt1}. 
%
%
\begin{lem} \mbox{} \label{lem:QXt}
\begin{enu}
\item We have $\SQ_{\beta,k}\SQ_{\pm \beta,l}=0$ 
for $\beta \in \Phi \setminus (\Pi \cup (-\Pi))$ and $k,l \in \mathbb{Z}$, 
where $\Pi=\{ \alpha_{i} \}_{i \in I}$ is the set of simple roots. 
%

\item We have $\SX^{\nu}\SX^{\mu}=\SX^{\mu+\nu}$ for $\mu,\nu \in P$. 

\item We have $\SQ_{\beta,k} \SX^{\nu} = \SX^{s_{\beta}\nu} \SQ_{\beta,k}$ 
for $\nu \in P$, $\beta \in \Phi$, and $k \in \mathbb{Z}$. 
\end{enu}
\end{lem}

We set
\begin{equation}
\SR_{\beta,k}:=\SX^{\rho} (\SX^{\beta}+\SQ_{\beta,k}) \SX^{-\rho} \qquad 
\text{for $\beta \in \Phi$ and $k \in \mathbb{Z}$}.
\end{equation}
Let
\begin{equation} \label{eq:Xi}
\Xi:A_{0} \edge{-\beta_{1}} A_{1} \edge{-\beta_{2}} \cdots \edge{-\beta_{m}} A_{m}
\end{equation}
be a sequence of adjacent alcoves (note that $A_{0}$ is not necessarily identical to $A_{\circ}$). 
For an arbitrary sequence of integers $\bk=(k_{1},k_{2},\dots,k_{m})$, 
we set
\begin{equation}
\SR_{\Xi,\bk} 
 := \SR_{\beta_{m},k_{m}}\SR_{\beta_{m-1},k_{m-1}} \cdots \SR_{\beta_{2},k_{2}}\SR_{\beta_{1},k_{1}}. 
\end{equation}
By the same argument as for \cite[Proposition~14.5]{LP}, 
we can prove the following proposition; 
notice that in the proof of \cite[Proposition~14.5]{LP}, 
they use only the commutation relations 
corresponding to those in Lemma~\ref{lem:QXt}(2),(3), together with
some facts about central points of alcoves 
\cite[Lemmas~14.1 and 14.2]{LP}.
%
%
\begin{prop} \label{prop:RXik}
Keep the notation and setting above. 
Then, for $u \in W$ and $\xi \in Q^{\vee,+}$, 
\begin{equation}
\SR_{\Xi,\bk} [ut_{\xi}] = 
\sum_{ A = \{j_{1},\dots,j_{s}\} } 
\be^{ -u\mu_{A} }
\SQ_{\beta_{j_{s}},k_{j_{s}}} \cdots \SQ_{\beta_{j_{2}},k_{j_{2}}}\SQ_{\beta_{j_{1}},k_{j_{1}}} [ut_{\xi}], 
\end{equation}
where $A = \{j_{1},\dots,j_{s}\}$ runs over all subsets 
of $[m]:=\{1,2,\dots,m\}$, and $\mu_{A} \in P$ is a weight 
depending only on $A$. In particular, if $A_{0}=A_{\circ}$ and 
$A_{m}=A_{-\lambda}$ for a weight $\lambda \in P$, 
then $\mu_{A} = \ha{r}_{j_{1}}\ha{r}_{j_{2}} \cdots \ha{r}_{j_{s}}(-\lambda)$ 
for $A = \{j_{1},\dots,j_{s}\} \subseteq [m]$. 
\end{prop}

\begin{dfn}
Let $\Xi$ be as in \eqref{eq:Xi}, and $u \in W$. 
A subset $A=\bigl\{ j_1 < j_2 < \cdots < j_s \bigr\}$ of $[m]=\{1,\ldots,m\}$ (possibly empty) is 
a $u$-admissible subset (with respect to $\Xi$) 
if there exists a directed path of the form \eqref{eqn:admissible} 
(with $w$ replaced by $u$) in the quantum Bruhat graph $\QB(W)${\rm ;} 
we define $\ed(u,A) \in W$ in the same manner as in \eqref{eqn:admissible}.
%
%
Let $\CA(u,\Xi)$ denote the collection of all $u$-admissible subsets of $[m]$. 
\end{dfn}

Let $\Xi$ be as in \eqref{eq:Xi}, and $u \in W$. 
For $A \in \CA(u,\Xi)$, we define 
$\wt(u,A)$, $A^{-}$, and $\dn(u,A)$ 
(resp., $n(A)$) in exactly the same way as in Section~\ref{sec:qam} 
(resp., Theorem~\ref{genchev}). 
Let $\bk=(k_{1},k_{2},\dots,k_{m})$ be an arbitrary sequence of integers. 
For $A=\bigl\{ j_1 < j_2 < \cdots < j_s \bigr\} \in \CA(u,\Xi)$, we set 
\begin{equation*}
\Ht_{\bk}(u,A):= \sum_{j \in A^{-}} \sgn(\beta_{j})k_{j}. 
\end{equation*}
Then the next corollary follows from Proposition~\ref{prop:RXik}, together with 
the definition of $\SQ_{\beta,k}$, and the definitions of 
$\CA(u,\Xi)$, $n(A)$, $\Ht_{\bk}(u,A)$, $\ed(u,A)$, $\dn(u,A)$, $\wt(u,A)$ above. 

\begin{cor} \label{cor:RXik}
Keep the notation and setting above. 
Then, 
\begin{equation}
\SR_{\Xi,\bk} [ut_{\xi}] = 
\sum_{A \in \CA(u,\Xi)} 
(-1)^{n(A)} q^{-\Ht_{\bk}(u,A)} \be^{ -u\mu_{A} } 
[\ed(u,A) t_{\xi+\dn(u,A)}]; 
\end{equation}
note that if $A_{0}=A_{\circ}$ and 
$A_{m}=A_{-\lambda}$ for some weight $\lambda \in P$ (that is, 
if $\Xi$ is a $\lambda$-chain), then $-u\mu_{A} = \wt(u,A)$. 
\end{cor}


Now, let $\alpha,\beta \in \Phi$ be such that 
$\pair{\alpha}{\beta^{\vee}} \le 0$, or equivalently, 
$\pair{\beta}{\alpha^{\vee}} \le 0$. Let 
\begin{equation*}
\Xi:A_{0} \edge{-\beta_{1}} A_{1} \edge{-\beta_{2}} \cdots \edge{-\beta_{m}} A_{m}
\end{equation*}
be a sequence of adjacent alcoves (note that $A_{0}$ is not necessarily $A_{\circ}$) with 
\begin{equation*}
\beta_{1} = \alpha, \qquad 
\beta_{2} = s_{\alpha} \beta, \qquad 
\beta_{3} = s_{\alpha}s_{\beta} \alpha, 
\quad \ldots, \quad
\beta_{m-1} = s_{\beta} \alpha, \quad 
\beta_{m} = \beta. 
\end{equation*}
Then we have a sequence of adjacent alcoves of the form: 
\begin{equation*}
\Theta : A_{0} = B_{0} \edge{-\gamma_{1}} B_{1} \edge{-\gamma_{2}} \cdots \edge{-\gamma_{m}} A_{m}=B_{m}, 
\end{equation*}
where 
\begin{equation*}
\gamma_{1} = \beta, \qquad 
\gamma_{2} = s_{\beta} \alpha, 
\quad \ldots, \quad
\gamma_{m-2} = s_{\alpha}s_{\beta} \alpha, \qquad 
\gamma_{m-1} = s_{\alpha} \beta, \qquad 
\gamma_{m} = \alpha. 
\end{equation*}
%
%
\begin{prop} \label{prop:YB}
Let $\Xi$ and $\Theta$ be as above. 
Assume that $\bk=(k_{1},k_{2},\dots,k_{m})$ and 
$\bl=(l_{1},l_{2},\dots,l_{m})$ are sequences of integers satisfying 
the condition that 
\begin{equation} \label{eq:is}
\Biggl(\bigcap_{p=1}^{m} H_{\beta_{p},k_{p}}\Biggr) \cap 
\Biggl(\bigcap_{p=1}^{m} H_{\gamma_{p},l_{p}}\Biggr) \ne \emptyset.
\end{equation}
Then the equality $\SR_{\Xi,\bk} = \SR_{\Theta,\bl}$ holds. 
\end{prop}

In the proof of Proposition~\ref{prop:YB}, we use the following.
%
%
\begin{lem} \label{lem:YB1}
Keep the notation and setting of Proposition~\ref{prop:YB}. 
Let $A \in \CA(u,\Xi)$. If $B \in \CA(u,\Xi)$ (resp., $B \in \CA(u,\Theta)$) 
satisfies $\dn(u,A)=\dn(u,B)$, then $\Ht_{\bk}(u,A)=\Ht_{\bk^{B}}(u,B)$, 
where $\bk^{B}:=\bk$ (resp., $\bk^{B}:=\bl$). 
\end{lem}

\begin{proof}
If $B \in \CA(u,\Xi)$ (resp., $B \in \CA(u,\Theta)$), then 
we set $\beta_{p}^{B}:=\beta_{p}$ (resp., $\beta_{p}^{B}:=\gamma_{p}$)
for $1 \le p \le m$. We have 
\begin{equation} \label{eq:dn1}
\begin{split}
\sum_{a \in A^{-}} \sgn(\beta_{p})\beta_{p}^{\vee} & = 
\sum_{a \in A^{-}} |\beta_{a}|^{\vee} =  \dn(u,A) = \dn(u,B) \\[2mm]
& = \sum_{b \in B^{-}} |\beta_{b}^{B}|^{\vee} 
  = \sum_{b \in B^{-}} \sgn(\beta_{p}^{B})(\beta_{p}^{B})^{\vee}. 
\end{split}
\end{equation}
Let us take an element $\mu$ in the (non-empty) intersection \eqref{eq:is}. 
Then we have $\pair{\mu}{\beta_{p}^{\vee}} = k_{p}$ for $1 \le p \le m$. 
Also, if we write $\bk^{B}$ as $\bk^{B}=(k_{1}^{B},k_{2}^{B},\dots,k_{m}^{B})$, 
then $\pair{\mu}{(\beta_{p}^{B})^{\vee}} = k_{p}^{B}$ for $1 \le p \le m$. 
Therefore, we see that 
\begin{equation*}
\begin{split}
\Ht_{\bk}(u,A) & = \sum_{a \in A^{-}} \sgn(\beta_{a}) k_{a} 
= \sum_{a \in A^{-}} \sgn(\beta_{a}) \pair{\mu}{\beta_{a}^{\vee}} \\[2mm]
& \stackrel{\eqref{eq:dn1}}{=} 
\sum_{b \in B^{-}} \sgn(\beta_{b}^{B}) \pair{\mu}{(\beta_{b}^{B})^{\vee}} =
\sum_{b \in B^{-}} \sgn(\beta_{b}^{B}) k_{b}^{B} = \Ht_{\bk^{B}}(u,B), 
\end{split}
\end{equation*}
as desired.
\end{proof}

\begin{proof}[Proof of Proposition~{\rm \ref{prop:YB}}]
We show that 
$\SR_{\Xi,\bk}[ut_{\xi}] = \SR_{\Theta,\bl}[ut_{\xi}]$ for each $u \in W$ 
and $\xi \in Q^{\vee,+}$. Fix $u \in W$ and $\xi \in Q^{\vee,+}$ arbitrarily, 
and write $\SR_{\Xi,\bk}[ut_{\xi}]$ and $\SR_{\Theta,\bl}[ut_{\xi}]$ as: 
\begin{equation*}
\SR_{\Xi,\bk}[ut_{\xi}] = \sum_{v \in W,\,\zeta \in Q^{\vee,+}}
a_{v,\zeta}(q)[vt_{\zeta}], \qquad 
\SR_{\Theta,\bl}[ut_{\xi}] = \sum_{v \in W,\,\zeta \in Q^{\vee,+}}
b_{v,\zeta}(q)[vt_{\zeta}], 
\end{equation*}
where $a_{v,\zeta}(q)$ and $b_{v,\zeta}(q)$ are elements of 
$\mathbb{Z}[q,q^{-1}][P]$; it suffices to show that $a_{v,\zeta}(q) = b_{v,\zeta}(q)$ 
for all $v \in W$ and $\zeta \in Q^{\vee,+}$. 
By Corollary~\ref{cor:RXik}, we have for $v \in W$ and $\zeta \in Q^{\vee,+}$, 
\begin{align*}
a_{v,\zeta}(q) & = 
\sum_{ \begin{subarray}{c} 
A \in \CA(u,\Xi) \\
\ed(u,A)=v, \, \xi+\dn(u,A)=\zeta
\end{subarray} } (-1)^{n(A)} q^{-\Ht_{\bk}(u,A)} \be^{ -u\mu_{A} }, \\[2mm]
b_{v,\zeta}(q) & = 
\sum_{ \begin{subarray}{c} 
A \in \CA(u,\Theta) \\
\ed(u,A)=v, \, \xi+\dn(u,A)=\zeta
\end{subarray} } (-1)^{n(A)} q^{-\Ht_{\bl}(u,A)} \be^{ -u\mu_{A} }.
\end{align*}
From Lemma~\ref{lem:YB1}, we see that 
the function $A \mapsto \Ht_{\bk}(u,A)$ is constant 
on the subset $\bigl\{A \in \CA(u,\Xi) \mid \ed(u,A)=v,\,\xi+\dn(u,A)=\zeta\bigr\}$. 
Hence it follows that 
\begin{equation} \label{eq:C}
a_{v,\zeta}(q) = q^{C_{v,\zeta}}
\sum_{ \begin{subarray}{c} 
A \in \CA(u,\Xi) \\
\ed(u,A)=v, \, \xi+\dn(u,A)=\zeta
\end{subarray} } (-1)^{n(A)}  \be^{ -u\mu_{A} } = 
q^{C_{v,\zeta}}a_{v,\zeta}(1)
\end{equation}
for some integer $C_{v,\zeta} \in \mathbb{Z}$. 
Similarly, we deduce that 
\begin{equation} \label{eq:D}
b_{v,\zeta}(q) = q^{D_{v,\zeta}}
\sum_{ \begin{subarray}{c} 
A \in \CA(u,\Theta) \\
\ed(u,A)=v, \, \xi+\dn(u,A)=\zeta
\end{subarray} } (-1)^{n(A)}  \be^{ -u\mu_{A} } = 
q^{D_{v,\zeta}}b_{v,\zeta}(1)
\end{equation}
for some integer $D_{v,\zeta} \in \mathbb{Z}$. 
Here we see from Proposition~\ref{prop:YBq1} that 
$a_{v,\zeta}(1) = b_{v,\zeta}(1)$; 
note that the specialization of the operator $\SQ_{\beta,k}$ at $q=1$ 
is identical to $\SQ_{\beta}$ given by \eqref{eq:qbo1}, 
and hence the specialization of the operator $\SR_{\beta,k}$ at $q=1$ 
is identical to $\SR_{\beta}$ given by \eqref{eq:sr1}. Therefore, we find that 
\begin{equation*}
a_{v,\zeta}(q) = 0 \iff b_{v,\zeta}(q) = 0. 
\end{equation*}
Hence it remains to show that if 
$a_{v,\zeta}(q) \ne 0$, or equivalently, 
if $b_{v,\zeta}(q) \ne 0$, then $C_{v,\zeta}=D_{v,\zeta}$; 
notice that in this case, 
\begin{align*}
& \bigl\{A \in \CA(u,\Xi) \mid \ed(u,A)=v,\,\xi+\dn(u,A)=\zeta\bigr\} \ne \emptyset, \\
& \bigl\{A \in \CA(u,\Theta) \mid \ed(u,A)=v,\,\xi+\dn(u,A)=\zeta\bigr\} \ne \emptyset.
\end{align*}
Also, we deduce from Lemma~\ref{lem:YB1} that if $A \in \CA(u,\Xi)$ and 
$B \in \CA(u,\Theta)$ satisfy $\dn(u,A)=\dn(u,B)$, then 
$\Ht_{\bk}(u,A)=\Ht_{\bl}(u,B)$. From these, we obtain $C_{v,\zeta}=D_{v,\zeta}$. 
This completes the proof of Proposition~\ref{prop:YB}. 
\end{proof}

\subsection{Proof of Theorem~{\rm \ref{genchev}}}
Fix $w \in W$. Let $\lambda \in P$ be an arbitrary weight, and let 
\begin{equation} \label{eq:Gam}
\Gamma: 
A_{\circ}=A_{0} \edge{-\beta_{1}} A_{1} \edge{-\beta_{2}} \cdots \edge{-\beta_{m}} A_{m}=A_{-\lambda}
\end{equation}
be an arbitrary (not necessarily reduced) $\lambda$-chain of roots, 
where $A_{-\lambda} =A_{\circ}-\lambda$. Let $H_{\beta_{i},-l_{i}}$ be 
the common wall of $A_{i-1}$ and $A_{i}$ for $i=1,2,\dots,m$. We set 
\begin{equation} \label{eq:til}
\ti{l}_{i}:=\pair{\lambda}{\beta_{i}^{\vee}}-l_{i}
\end{equation}
for $i=1,2,\dots,m$, and then 
$\ti{\bl}:=(\ti{l}_{1},\ti{l}_{2},\dots,\ti{l}_{m})$; 
note that $\Ht(w,A) = \Ht_{\ti{\bl}}(w,A)$ for $A \in \CA(w,\Gamma)$. If we set 
%
%
\begin{equation} \label{eq:G}
\begin{split}
& \bG_{\Gamma}(w,\xi):= \\[2mm]
& \sum_{\bchi \in \ol{\Par(\lambda)}}\,
  \sum_{A \in \CA(w,\Gamma)} 
  (-1)^{n(A)}q^{-\Ht(w,A)-|\bchi|-\pair{\lambda}{\xi}} 
  \be^{\wt(w,A)} [ \ed(w,A)t_{\xi+\dn(w,A)+\iota(\bchi)} ]
\end{split}
\end{equation}
for $\xi \in Q^{\vee,+}$, 
then we see by Corollary~\ref{cor:RXik} that 
\begin{equation} \label{eq:G2}
\bG_{\Gamma}(w,\xi) = \sum_{\bchi \in \ol{\Par(\lambda)}} 
  q^{-|\bchi|-\pair{\lambda}{\xi}} \SR_{\Gamma,\ti{\bl}}[wt_{\xi+\iota(\bchi)}]. 
\end{equation}

Let $\Gamma$ be as in \eqref{eq:Gam}. 
Let $\alpha,\beta \in \Phi$ be such that 
$\pair{\alpha}{\beta^{\vee}} \le 0$, or equivalently, 
$\pair{\beta}{\alpha^{\vee}} \le 0$. 
Assume that there exist $1 \le u < t \le m$ such that 
\begin{equation} \label{eq:YBa}
\beta_{u}= \alpha,\quad \beta_{u+1}=s_{\alpha} \beta, \quad 
\beta_{u+2} = s_{\alpha}s_{\beta} \alpha, \quad \dots, \quad 
\beta_{t-1} = s_{\beta} \alpha, \quad 
\beta_{t} = \beta; 
\end{equation}
we set 
\begin{equation} \label{eq:XYZ}
X:=\bigl\{1,2,\dots,u-1\bigr\}, \quad 
Y:=\bigl\{u,u,\dots,t-1,t\bigr\}, \quad 
Z:=\bigl\{t+1,t+2,\dots,m\bigr\}.
\end{equation}
Let
\begin{equation} \label{eq:Gam2}
\Gamma': A_{\circ}=B_{0} \edge{-\gamma_{1}} B_{1} \edge{-\gamma_{2}} \cdots
        \edge{-\gamma_{m-1}} B_{m-1} \edge{-\gamma_{m}} B_{m}=A_{-\lambda}
\end{equation}
be the $\lambda$-chain obtained by applying the procedure (YB) in Remark~\ref{rem:chain} to 
\begin{equation*}
(\beta_{u},\beta_{u+1},\dots,\beta_{t-1},\beta_{t})
\end{equation*}
in $\Gamma$; that is, $\gamma_{p}=\beta_{p}$ for all $p \in X \cup Z$, and 
\begin{equation} \label{eq:YB2}
\begin{split}
(\gamma_{u},\gamma_{u+1},\dots,\gamma_{t-1},\gamma_{t}) & = 
(\beta_{t},\beta_{t-1},\dots,\beta_{u+1},\beta_{u}) \\
& = (\beta, s_{\beta} \alpha, \dots, s_{\alpha}s_{\beta} \alpha, s_{\alpha} \beta, \alpha). 
\end{split}
\end{equation}
Let $H_{\gamma_{i},-k_{i}}$ be 
the common wall of $B_{i-1}$ and $B_{i}$ for $i=1,2,\dots,m$. We set 
$\ti{k}_{i}:=\pair{\lambda}{\gamma_{i}^{\vee}}-k_{i}$ 
for $i=1,2,\dots,m$, and then $\ti{\bk}:=(\ti{k}_{1},\ti{k}_{2},\dots,\ti{k}_{m})$; 
note that $\Ht(w,B) = \Ht_{\ti{\bk}}(w,B)$ for $B \in \CA(w,\Gamma')$. 
%
%
\begin{prop} \label{prop:YB2}
Keep the notation and setting above. 
Then, $\SR_{\Gamma,\ti{\bl}} = \SR_{\Gamma',\ti{\bk}}$, and 
$\bG_{\Gamma}(w,\xi)=\bG_{\Gamma'}(w,\xi)$ for all $\xi \in Q^{\vee,+}$. 
\end{prop}

\begin{proof}
We see from (the last sentence of) \cite[Lemma 5.3]{LP} that 
the sequences of hyperplanes $H_{\beta_{i},-l_{i}}$, $i=1,2,\dots,m$, 
and $H_{\gamma_{i},-k_{i}}$, $i=1,2,\dots,m$, coincide, except that 
the segments corresponding to $i=u,u+1,\dots,t-1,t$ are reversed. 
It follows from \cite[Lemma~3.5]{lalurc} that 
\begin{equation*}
\Biggl(\bigcap_{p=u}^{t} H_{\beta_{p},-l_{p}}\Biggr) \cap 
\Biggl(\bigcap_{p=u}^{t} H_{\gamma_{p},-k_{p}}\Biggr) \ne \emptyset.
\end{equation*}
If $\mu$ is an element of this (non-empty) intersection, then $-\lambda-\mu$ is 
an element of the intersection
\begin{equation*}
\Biggl(\bigcap_{p=u}^{t} H_{\beta_{p},-\ti{l}_{p}}\Biggr) \cap 
\Biggl(\bigcap_{p=u}^{t} H_{\gamma_{p},-\ti{k}_{p}}\Biggr);
\end{equation*}
in particular, this intersection is non-empty. 
Therefore, by applying Proposition~\ref{prop:YB} 
to the subproducts in $\SR_{\Gamma,\ti{\bl}}$ and 
$\SR_{\Gamma',\ti{\bk}}$ corresponding to the subset $Y$ of $[m]$ 
(i.e., the parts changed by the procedure (YB)), we deduce that 
$\SR_{\Gamma,\ti{\bl}} = \SR_{\Gamma',\ti{\bk}}$. 
Therefore, we obtain
\begin{align*}
\bG_{\Gamma}(w,\xi) & = 
\sum_{\bchi \in \ol{\Par(\lambda)}} 
  q^{-|\bchi|-\pair{\lambda}{\xi}} \SR_{\Gamma,\ti{\bl}}[wt_{\xi+\iota(\bchi)}] 
= \sum_{\bchi \in \ol{\Par(\lambda)}} 
  q^{-|\bchi|-\pair{\lambda}{\xi}} \SR_{\Gamma',\ti{\bk}}[wt_{\xi+\iota(\bchi)}] \\
& = \bG_{\Gamma'}(w,\xi), 
\end{align*}
as desired. 
\end{proof}

Let $\lambda \in P$ be as above. 
We set $\SRL{\lambda}_{q}:=\SR_{\Gamma,\ti{\bl}}$, 
with $\Gamma$ a reduced $\lambda$-chain and 
$\ti{\bl}$ given by \eqref{eq:til}; 
by Proposition~\ref{prop:YB2} and Remark~\ref{rem:chain}, 
we see that the operator $\SRL{\lambda}_{q}$ 
does not depend on the choice of a reduced $\lambda$-chain $\Gamma$. 
For simplicity of notation, 
we write $\OQGl{\nu}$ as $[\nu]$ for $\nu \in P$. 
%
%
\begin{thm} \label{thm:main}
Let $x=wt_{\xi} \in W_{\af}^{\ge 0}$, with $w \in W$ and $\xi \in Q^{\vee,+}$. 
Let $\lambda \in P$ be an arbitrary weight, 
and let $\Gamma$ be an arbitrary reduced $\lambda$-chain. Then, 
%
%
\begin{equation} \label{eq:main}
[ -w_{\circ} \lambda ] \cdot [x] = 
\sum_{\bchi \in \ol{\Par(\lambda)}} 
  q^{-|\bchi|-\pair{\lambda}{\xi}} \SRL{\lambda}_{q}[wt_{\xi+\iota(\bchi)}]. 
\end{equation}
\end{thm}

\begin{proof}
If $\lambda$ is a dominant (resp., anti-dominant) weight, 
then equation \eqref{eq:main} follows from Theorem~\ref{chevdomqam}
(resp., Theorem~\ref{chevantidomqam}) and \eqref{eq:G2}, 
together with the fact that the operator $\SRL{\lambda}_{q}$ 
does not depend on the choice of a reduced $\lambda$-chain $\Gamma$; 
recall that the lex $\lambda$-chain $\Gamma_{\mathrm{lex}}(\lambda)$ is a reduced $\lambda$-chain. 

Now, let $\lambda \in P$. 
Then, $\lambda = \lambda^{+} + \lambda^{-}$, where 
\begin{equation*}
\lambda^{+}:=\sum_{i \in I} \max\bigl( \pair{\lambda}{\alpha_{i}^{\vee}},\,0\bigr) \varpi_{i}, \qquad 
\lambda^{-}:=\sum_{i \in I} \min\bigl( \pair{\lambda}{\alpha_{i}^{\vee}},\,0\bigr) \varpi_{i};
\end{equation*}
note that $\lambda^{+}$ is dominant and 
$\lambda^{-}$ is anti-dominant. 
Let $\Gamma^{\pm}$ be reduced $\lambda^{\pm}$-chains, respectively, 
and write them as: 
\begin{equation*}
\Gamma^{+}: A_{\circ}=A_{0}' \edge{-\beta_{1}'} \cdots 
          \edge{-\beta_{m'}'} A_{m'}'=A_{-\lambda^{+}}, 
\end{equation*}
\begin{equation*}
\Gamma^{-}: A_{\circ}=A_{0}'' \edge{-\beta_{1}''} \cdots 
          \edge{-\beta_{m''}''} A_{m''}''=A_{-\lambda^{-}}; 
\end{equation*}
we have $\beta_{i}' \in \Phi^{+}$ for all $1 \le i \le m'$, and 
$\beta_{i}'' \in \Phi^{-}$ for all $1 \le i \le m''$. 
Let $H_{\beta_{i}',-l_{i}'}$ be the common wall of $A_{i-1}'$ and $A_{i}'$ 
for $i=1,2,\dots,m'$, and let $H_{\beta_{i}'',-l_{i}''}$ be 
the common wall of $A_{i-1}''$ and $A_{i}''$ for $i=1,2,\dots,m''$.
Let $\Gamma_{0}$ be 
the concatenation of $\Gamma^{+}$ and $\Gamma^{-}$, that is, 
\begin{equation*}
\Gamma_{0}: \underbrace{ A_{\circ}=A_{0} \edge{-\beta_{1}} \cdots \edge{-\beta_{m'}} A_{m'}}_{\Gamma^{+}}
=
\underbrace{A_{-\lambda^{+}} \edge{-\beta_{m'+1}} \cdots \edge{-\beta_{m}} A_{m} = A_{-\lambda} }_{%
\Gamma^{-} \text{ (shifted by $-\lambda^{+}$)} }, 
\end{equation*}
where $m=m'+m''$, and 
\begin{align}
A_{i} & =
 \begin{cases}
 A_{i}' & \text{for $0 \le i \le m'$}, \\[1.5mm]
 A_{i-m'}''-\lambda^{+} & \text{for $m' \le i \le m=m'+m''$},
 \end{cases} \label{eq:A} \\[3mm]
\beta_{i} & =
 \begin{cases}
 \beta_{i}' & \text{for $0 \le i \le m'$}, \\[1.5mm]
 \beta_{i-m'}'' & \text{for $m' \le i \le m=m'+m''$}.
 \end{cases} \label{eq:beta}
\end{align}
If we denote by $H_{\beta_{i},-l_{i}}$ the common wall of $A_{i-1}$ and $A_{i}$ 
for $i=1,2,\dots,m$, then 
\begin{equation} \label{eq:l}
l_{i}=
 \begin{cases}
 l_{i}' & \text{for $0 \le i \le m'$}, \\[1.5mm]
 l_{i-m'}''+ \pair{\lambda^{+}}{(\beta_{i-m'}'')^{\vee}} & \text{for $m' \le i \le m=m'+m''$}.
 \end{cases}
\end{equation}

We will show that 
\begin{equation}
[ -w_{\circ} \lambda ] \cdot [ x ] = \bG_{\Gamma_{0}}(w,\xi). 
\end{equation}
From Theorems~\ref{chevdomqam} and \ref{chevantidomqam}, we see that 
\begin{align}
& [ -w_{\circ} \lambda ] \cdot [ x ]
  = [ -w_{\circ} \lambda^{-} ] \cdot [ -w_{\circ} \lambda^{+} ] \cdot [ x ] \label{eq:G0-1} \\[3mm]
& = \sum_{A \in \CA(w,\Gamma^{+})}
     \sum_{ \bchi \in \ol{\Par(\lambda^{+})} }
     q^{-\Ht(w,A)-\pair{\cri{\lambda^+}}{\xi}-|\bchi|} \be^{\wt(w,A)}
     [ -w_{\circ} \lambda^{-} ] \cdot [ \ed(w,A)t_{\xi+\dn(w,A)+\iota(\bchi)} ] \nonumber \\[3mm]
& = \sum_{A \in \CA(w,\Gamma^{+})}\sum_{B \in \CA(\ed(w,A),\Gamma^{-})}
    \sum_{ \bchi \in \ol{\Par(\lambda^{+})} } (-1)^{|B|} \nonumber \\
& \hspace*{25mm} \times 
    q^{-\Ht(w,A)-\pair{\lambda^{+}}{\xi}-|\bchi|-\Ht(\cri{{\rm end}(w,A)},B)-\pair{\lambda^{-}}{\xi+\dn(w,A)+\iota(\bchi)} } \nonumber \\
& \hspace*{25mm} \times \be^{\wt(w,A)+\wt(\cri{{\rm end}(w,A)},B)} [ \ed(\cri{{\rm end}(w,A)},B)t_{\xi+\dn(w,A)+\iota(\bchi)+\dn(\cri{{\rm end}(w,A)},B)} ]; \nonumber 
\end{align}
note that $\pair{\lambda^{-}}{\iota(\bchi)}=0$ and $\ol{\Par(\lambda^{+})} = \ol{\Par(\lambda)}$. 
We have a natural bijection from the set $\bigl\{ (A,B) \mid A \in \CA(w,\Gamma^{+}),\,
B \in \CA(\ed(w,A),\Gamma^{-})\bigr\}$ onto $\CA(w,\Gamma_{0})$ given by 
concatenating $A \in \CA(w,\Gamma^{+})$ with $B \in \CA(\ed(w,A),\Gamma^{-})$, 
which we denote by $A \ast B$. In addition, it is easily verified that 
\begin{equation*}
\begin{split}
& n(A \ast B) = |B|, \quad 
  \dn(w,A)+\dn(\ed(w,A), B) = \dn (w, A \ast B), \\ 
& \ed(\ed(w,A),B)=\ed(w,A \ast B), 
\end{split}
\end{equation*}
and
\begin{align*}
& \Ht(w,A) + \Ht(\ed(w,A), B) + \pair{\lambda^{-}}{\dn(w,A)} \\[3mm]
& \qquad =
  \sum_{j \in A^{-}} \bigl(\pair{\lambda^{+}}{(\beta_{j}')^{\vee}}-l_{j}' \bigr) - 
  \sum_{j \in B^{-}} \bigl(\pair{\lambda^{-}}{(\beta_{j}'')^{\vee}}-l_{j}'' \bigr) + 
  \sum_{j \in A^{-}} \pair{\lambda^{-}}{(\beta_{j}')^{\vee}} \\[3mm]
& \qquad =
  \sum_{j \in A^{-}} \bigl( \pair{\lambda}{(\beta_{j}')^{\vee}}-l_{j}' \bigr) - 
  \sum_{j \in B^{-}} \bigl( \pair{\lambda}{(\beta_{j}'')^{\vee}}-
  \pair{\lambda^{+}}{(\beta_{j}'')^{\vee}}-l_{j}'' \bigr) \\[3mm]
& \qquad = \Ht(w,A \ast B) \qquad \text{by \eqref{eq:beta} and \eqref{eq:l}}. 
\end{align*}
On another hand, consider the galleries $\gamma(w,A)$ and $\gamma(\ed(w,A), B)+\wt(\gamma(w,A))$, which are constructed based on $\Gamma^+$ and $\Gamma^-$, respectively (cf. Section~\ref{sec:gal}). By Proposition~\ref{conc-gal}~(2), these galleries can be concatenated (cf. Definition~\ref{def-conc}). Moreover, by the construction of these galleries, we have
\[\gamma(w,A)\ast\left(\gamma(\ed(w,A), B)+\wt(\gamma(w,A))\right)=\gamma(w,A\ast B)\,,\]
where $\gamma(w,A\ast B)$ is constructed based on $\Gamma_0$. 
By considering the weights of the two sides, and by applying Proposition~\ref{conc-gal}~(1), we derive
\begin{equation*}
\wt(w, A)+\wt(\ed(w,A), B)=\wt(w, A \ast B)\,.
\end{equation*} 
We conclude that the right-hand side of \eqref{eq:G0-1} is identical to $\bG_{\Gamma_{0}}(w,\xi)$, as desired. 
Hence, by \eqref{eq:G2}, we have
\begin{equation} \label{eq:main1}
[ -w_{\circ} \lambda ] \cdot [ x ] = \bG_{\Gamma_{0}}(w,\xi) = 
\sum_{ \chi \in \ol{\Par(\lambda)} } 
q^{ -|\bchi|-\pair{\lambda}{\xi} } \SR_{\Gamma_{0},\ti{\bl}_{0}}[wt_{\xi+\iota(\bchi)}]\,,
\end{equation}
where $\ti{\bl}_{0}$ is given by \eqref{eq:til} for $\Gamma_{0}$ 
(see also \eqref{eq:l}). 

Now, let $\Gamma$ be an arbitrary reduced $\lambda$-chain, 
with $\ti{\bl}$ given by \eqref{eq:til} for this $\Gamma$.
Because the concatenation $\Gamma_{0}$ above of $\Gamma^{+}$ and $\Gamma^{-}$ 
is a $\lambda$-chain, there exists a sequence $\Gamma_{0}, \Gamma_{1},\dots,\Gamma_{s}=\Gamma$ 
of $\lambda$-chains such that $\Gamma_{t}$ is obtained from $\Gamma_{t-1}$
by applying either (YB) or (D) for each $t=1,2,\dots,s$ 
(see Remark~\ref{rem:chain}). For $t = 1,2,\dots,s$, 
let $\ti{\bl}_{t}$ be given by \eqref{eq:til} for $\Gamma_{t}$. 
We show that 
\begin{equation} \label{eq:main2}
\SR_{\Gamma_{t-1},\ti{\bl}_{t-1}} = \SR_{\Gamma_{t},\ti{\bl}_{t}}%
\qquad \text{for all $t=1,2,\dots,s$}.
\end{equation}
If $\Gamma_{t}$ is obtained from $\Gamma_{t-1}$
by applying (YB), then it follows from Proposition~\ref{prop:YB2} that 
$\SR_{\Gamma_{t-1},\ti{\bl}_{t-1}} = \SR_{\Gamma_{t},\ti{\bl}_{t}}$. 
Assume that $\Gamma_{t}$ is obtained from $\Gamma_{t-1}$ by applying (D). 

\begin{claim} \label{c:main}
For $0 \le u \le s$, the $\lambda$-chain 
$\Gamma_{u}$ does not contain both of the roots
$\alpha_{i}$ and $-\alpha_{i}$ for any $i \in I$.
\end{claim}

\noindent
{\it Proof of Claim~{\rm \ref{c:main}}. }
We show this claim by induction on $0 \le u \le s$. 
Assume that $u = 0$. 
Let $i \in I$, and assume that $\langle \lambda,\alpha_{i}^{\vee} \rangle > 0$; 
note that $\langle \lambda^{+},\alpha_{i}^{\vee} \rangle > 0$  and 
$\langle \lambda^{-},\alpha_{i}^{\vee} \rangle = 0$. 
We see from \eqref{ranges} (see also \cite[Lemma~6.2]{LP}) 
that $\Gamma^{+}$ contains $\alpha_{i}$, 
but does not contain $-\alpha_{i}$, and that $\Gamma^{-}$ contains neither 
$\alpha_{i}$ nor $-\alpha_{i}$. Hence the concatenation $\Gamma_{0}$ 
of $\Gamma^{+}$ and $\Gamma^{-}$ contains $\alpha_{i}$, 
but does not contain $-\alpha_{i}$. 
Similarly, if $\langle \lambda,\alpha_{i}^{\vee} \rangle < 0$ (resp., $=0$), 
then $\Gamma_{0}$ contains $-\alpha_{i}$, 
but does not contain $\alpha_{i}$ (resp., $\Gamma_{0}$ contains 
neither $\alpha_{i}$ nor $-\alpha_{i}$). 

Assume that $u > 0$. 
If $\Gamma_{u}$ is obtained from $\Gamma_{u-1}$ by applying (D), 
then it is obvious by our induction hypothesis (for $\Gamma_{u-1}$)
and the definition of (D) that $\Gamma_{u}$ does not contain 
both of the roots $\alpha_{i}$ and $-\alpha_{i}$ for any $i \in I$.
Assume that $\Gamma_{u}$ is obtained from $\Gamma_{u-1}$ by applying (YB). 
Then we deduce by the definition of (YB) that 
the roots appearing in $\Gamma_{u}$ are the same as 
those appearing in $\Gamma_{u-1}$. It follows from this fact and 
our induction hypothesis (for $\Gamma_{u-1}$) that $\Gamma_{u}$ 
does not contain both of the roots $\alpha_{i}$ and $-\alpha_{i}$ 
for any $i \in I$. This proves Claim~\ref{c:main}. \bqed

\smallskip

Now, let us show \eqref{eq:main2} in the case that 
$\Gamma_{t}$ is obtained from $\Gamma_{t-1}$ by applying (D). 
In this case, a product of the form 
$\SR_{\beta,k}\SR_{-\beta,k}$ for some $\beta \in \Phi$ and $k \in \mathbb{Z}$ 
appears (at the part corresponding to the part in $\Gamma_{t-1}$ deleted by (D)) 
in $\SR_{\Gamma_{t-1},\ti{\bl}_{t-1}}$. We deduce by Claim~\ref{c:main} that 
$\beta \not\in \Pi \cup (-\Pi)$. Hence it follows from Lemma~\ref{lem:QXt} that 
$\SR_{\beta,k}\SR_{-\beta,k}$ is the identity map. Therefore, we also obtain 
$\SR_{\Gamma_{t-1},\ti{\bl}_{t-1}} = \SR_{\Gamma_{t},\ti{\bl}_{t}}$ in this case.
Thus we have shown \eqref{eq:main2}, which implies that 
\begin{equation} \label{eq:main3}
\SR_{\Gamma_{0},\ti{\bl}_{0}}=
\SR_{\Gamma_{1},\ti{\bl}_{1}}= \dots = 
\SR_{\Gamma_{s},\ti{\bl}_{s}}=
\SR_{\Gamma,\ti{\bl}}=\SRL{\lambda}_{q}.
\end{equation}
Combining \eqref{eq:main1} and \eqref{eq:main3}, 
we obtain \eqref{eq:main}. 
This completes the proof of Theorem~\ref{thm:main}. 
\end{proof}

Theorem~\ref{genchev} 
follows from Theorem~\ref{thm:main} and \eqref{eq:G2}.

\section{The quantum $K$-theory of flag \revise{manifolds}}\label{qkgb}

Y.-P. Lee defined the (small) {\em quantum $K$-theory} of a smooth projective variety $X$, denoted by $QK(X)$ (see \cite{leeqkt}). This is a deformation of the ordinary $K$-ring of $X$, analogous to the relation between quantum cohomology and ordinary cohomology. The deformed product is defined in terms of certain generalizations of {\em Gromov-Witten invariants} (i.e., the structure constants in quantum cohomology), called {\em quantum $K$-invariants of Gromov-Witten type}. 

In order to describe the (small) $T$-equivariant quantum $K$-theory $QK_T(G/B)$, for the finite-dimensional flag manifold $G/B$, we associate a (Novikov) variable $Q_k$ to each simple coroot $\alpha_k^{\vee}$, with $k\in I=\{1,\ldots,r\}$, 
and let $\bZ[Q]:=\bZ[Q_1,\dots,Q_r]$, $\bZ[\![Q]\!] := \bZ[\![Q_1, \ldots, Q_{r}]\!]$; 
given $\xi=d_1 \alpha_1^\vee+\cdots+d_r\alpha_r^\vee$ in $Q^{\vee,+}$, let $Q^{\xi} :=Q_1^{d_1}\cdots Q_r^{d_r}$. 
Also, let $\bZ[P][Q]:=\bZ[P] \otimes \bZ[Q]$ and $\bZ[P][\![Q]\!] := \bZ[P] \otimes \bZ[\![Q]\!]$, where the group algebra $\mathbb{Z}[P]$ of $P$ was defined at the beginning of Section~\ref{csi}. 
We define $QK_{T}(G/B)$ to be the $\bZ[P][\![Q]\!]$-module 
$K_T(G/B) \otimes_{{\mathbb Z}[P]} \bZ[P][\![Q]\!]$ ($\supset K_{T}(G/B) \otimes_{\mathbb{Z}[P]} \bZ[P][Q]$) equipped with the quantum multiplication, denoted by $\cdot$,
where $K_{T}(G/B)$ denotes the ordinary $T$-equivariant $K$-theory of $G/B$. 
The algebra $QK_T(G/B)$ has a $\mathbb{Z}[P][\![Q]\!]$-basis given by the classes $[{\mathcal O}^w]$, $w \in W$, of the structure sheaf of the (opposite) Schubert variety $X^{w} \subset G/B$ of codimension $\ell(w)$. 

\subsection{Main results}
 It is proved in \cite{kat1} that there exists a $\mathbb{Z}[P]$-module isomorphism from $QK_T(G/B)$ onto $K_{T}(\QQ)$ that respects the quantum multiplication in $QK_T(G/B)$ and the tensor product in $K_{T}(\QQ)$. 
More precisely, the isomorphism respects the quantum multiplication in $QK_{T}(G/B)$ with the line bundle class $[\mathcal{O}_{G/B}(-\varpi_{k})]$ and the tensor product in $K_{T}(\QQ)$ with the line bundle class $\OQGl{w_{\circ} \varpi_{k}}$, for all $k \in I$; 
in our notation, the line bundle $\mathcal{O}_{G/B}(- \nu)$ over $G/B$ for $\nu \in P$ denotes the $G$-equivariant line bundle constructed as the quotient space $G \times^{B} \mathbb{C}_{\nu}$ of the product space $G \times \mathbb{C}_{\nu}$ by the usual (free) left action of $B$, given by $b.(g, u) := (gb^{-1}, b u)$ for $b \in B$ and $(g, u) \in G \times \mathbb{C}_{\nu}$, where $\mathbb{C}_{\nu}$ is the one-dimensional $B$-module of weight $\nu \in P$. 
Here we remark that in order to translate the Chevalley formula in $K_{T}(\QQ)$ for fundamental weights into the one in the quantum $K$-theory of $G/B$, we need to consider $K_{T}(G/B) \otimes_{\mathbb{Z}[P]} \bZ[P][\![Q]\!]$ ($\supset K_{T}(G/B) \otimes_{\mathbb{Z}[P]} \bZ[P][Q]$);
for example, in type $A_{r}$, the tensor product in $K_{T}(\QQ)$ with the line bundle class $[\mathcal{O}_{\QQ}(-w_{\circ} \varepsilon_{k})]$ for $1 \leq k \leq r+1$ corresponds to the quantum multiplication with the class $\frac{1}{1-Q_{k}}[\mathcal{O}_{G/B}(\varepsilon_{k})]$, where $\varepsilon_{k} := \varpi_{k} - \varpi_{k-1}$, with $\varpi_{0} := 0$, $\varpi_{r+1} := 0$, and $Q_{r+1} := 0$ (see \cite[Section~5]{MNS} for details). 
Also, note that the isomorphism above sends each (opposite) Schubert class $[{\mathcal O}^w] Q^{\xi}$ (multiplied by a monomial $Q^{\xi}$ in the Novikov variables) in $QK_T(G/B)$ to the semi-infinite Schubert class $\OQG{wt_{\xi}}$ in $K_{T}(\QQ)$ for $w \in W$ and $\xi \in Q^{\vee,+}$; 
\revise{in our notation, this map sends $\mathbf{e}^{\mu} [\mathcal{O}^{w}] Q^{\xi}$ to $\mathbf{e}^{-\mu} \OQG{w t_{\xi}}$ for $w \in W$, $\xi \in Q^{\vee,+}$, and $\mu \in P$.} 
These results and the special case that $\lambda = - \varpi_{k}$ of the formula in Theorem~\ref{chevantidomqam} imply the Chevalley formula for $QK_T(G/B)$, stated below. We also use the well-known fact that  $[{\mathcal O}^{s_k}]=1-\mathbf{e}^{-\varpi_k}[\mathcal{O}_{G/B}(-\varpi_{k})]$ in $QK_T(G/B)$.

\begin{thm}\label{qkchev} Let $k \in I$, and fix a reduced $(-\varpi_k)$-chain $\Gamma(-\varpi_k)$. Then, in $QK_{T}(G/B)$, we have \revise{for $w \in W$}, 
\begin{align}\label{qkchev-f}&[{\mathcal O}^{s_k}]\cdot [{\mathcal O}^w] = \\[3mm]&\quad=
(1-\mathbf{e}^{w(\varpi_k) - \varpi_k})  [{\mathcal O}^w] +
\sum_{A\in{\mathcal A}(w,\Gamma(-\varpi_k))\setminus\{\emptyset\}}
(-1)^{|A|-1} \,Q^{{\rm{down}}(w,A)}\mathbf{e}^{-\varpi_k-{\rm{wt}}(w,A)} [{\mathcal O}^{{\rm end}(w,A)}]\,.\nonumber\end{align}
\end{thm}

\begin{rema}{\rm 
The non-equivariant version of \eqref{qkchev-f} (obtained by setting all equivariant coefficients $\mathbf{e}^\gamma$ to 1) was conjecturally stated in a slightly different form as \cite[Conjecture~17.1]{LP}, which we now explain. The {\em quantum Bruhat operators} defined in~\cite{BFP} were used. These are operators $Q_\beta$ indexed by positive roots $\beta$, which are defined on the group algebra of the Weyl group $W$ over $\bZ[Q]$; the action of $Q_\beta$ on $w\in W$ corresponds to the edge $w \xrightarrow{\beta} ws_{\beta}$ of $\QB(W)$ (if we do not have this edge in $\QB(W)$, then we define $Q_\beta(w):=0$). Let the reduced $(-\varpi_k)$-chain in Theorem~\ref{qkchev} be $(\beta_1,\ldots,\beta_m)$, and note that its reverse is a reduced $\varpi_k$-chain. The formula in  \cite[Conjecture~17.1]{LP} was expressed via the action of the operator 
\[1-(1-Q_{\beta_m})\cdots(1-Q_{\beta_1})\,.\]
By expanding the above product and acting on $w$, we obtain an alternating sum of $Q_{\beta_{j_s}}\cdots Q_{\beta_{j_1}}(w)$ for $w$-admissible subsets $\{j_1<\cdots<j_s\}$ in ${\mathcal A}(w,\Gamma(-\varpi_k))\setminus\{\emptyset\}$. This gives precisely the non-equivariant version of~\eqref{qkchev-f}.}
\end{rema}

Let us now turn to the type $A_{n-1}$ flag manifold $Fl_{n} := SL_{n}/B$ and its (non-equivariant) quantum $K$-theory $QK(Fl_{n}) = K(Fl_{n}) \otimes \mathbb{Z}[\![Q]\!]$, where $\mathbb{Z}[\![Q]\!] = \mathbb{Z}[\![Q_1, \ldots, Q_{n-1}]\!]$. In~\cite{lamqgp}, the first author and Maeno defined the so-called {\em quantum Grothendieck polynomials}, denoted by $\mathfrak{G}^{Q}_{w}$ for $w \in S_{n}$; the quantum Grothendieck polynomial $\mathfrak{G}_{w}^{Q}$ is defined to be the image of the classical Grothendieck polynomial $\mathfrak{G}_{w}$ under the $K$-theoretic quantization map $\widehat{Q}$ for each $w \in S_{n}$.
According to the Monk-type multiplication formula (i.e., \cite[Theorem~6.4]{lamqgp}) for quantum Grothendieck polynomials, whose proof is based on intricate combinatorics, these polynomials multiply precisely as stated by (the non-equivariant version of) the quantum $K$-Chevalley formula~\eqref{qkchev-f}; 
\revise{note that in the ordinary non-equivariant $K$-theory $K(Fl_{n})$, the (opposite) Schubert class $[\mathcal{O}^{w}]$ is identical to the class of the structure sheaf of the Schubert variety $X_{w_{\circ} w} \subset Fl_{n}$ of codimension $\ell(w)$ for each $w \in W = S_{n}$.} 

Let us add an explanation about the coincidence between our quantum $K$-Chevalley formula~\eqref{qkchev-f} for (opposite) Schubert classes and the Monk-type multiplication formula for quantum Grothendieck polynomials. Note first that the order in which the transpositions are applied in \cite[Theorem~6.4]{lamqgp} is precisely the one given by the reduced $(-\varpi_k)$-chain $\Gamma'(-\varpi_k)$ defined in Section~\ref{sec:qk-coeff-type-A}, i.e., the reverse of the chain~\eqref{omegak-chain}. Having observed this, the difference between the two formulas consists of the fact that the former is based on the Weyl group $S_n$, while the latter is based on the infinite symmetric group $S_\infty$. We address this difference below.

In view of the conjectural presentation (cf.~\cite[Theorem~3.10]{lamqgp}) of $QK(Fl_{n})$ proposed by Kirillov and Maeno, we consider the quotient ring $\bZ[\![Q]\!][x]/\widehat{I}^{Q}_{n}$; here $\bZ[\![Q]\!][x]$ denotes $\bZ[\![Q_1, \ldots, Q_{n-1}]\!][x_1, \ldots, x_n]$, and $\widehat{I}^{Q}_{n}$ is the ideal of $\bZ[\![Q]\!][x]$ generated by the specialization at $Q_{n} = 0$ of the images $\widehat{E}^{n}_{k}$ under the $K$-theoretic quantization map $\widehat{Q}$ of the elementary symmetric polynomials $e^{n}_{k}$, $1 \leq k \leq n$, where $e^{n}_{k}$ denotes the elementary symmetric polynomial of degree $k$ in the variables $x_1, \ldots, x_n$ (for details, see \cite[Section~3]{lamqgp}).
Note that our quantum $K$-Chevalley formula is in terms of the quantum Bruhat graph $\QB(S_{n})$ on the $n$-th symmetric group $S_{n}$, while the Monk-type multiplication formula is in terms of the quantum Bruhat graph $\QB(S_{\infty})$ on the infinite symmetric group $S_{\infty} = \bigcup_{m=1}^{\infty} S_{m}$. 
Hence, in order to prove the coincidence between them, we need to show that if $w \in S_{n+1}$ but $w \notin S_{n}$, then the associated quantum Grothendieck polynomial $\mathfrak{G}^{Q}_{w} = \widehat{Q}(\mathfrak{G}_{w})$ lies in the (defining) ideal $\widehat{I}^{Q}_{n}$ for the quotient ring $\bZ[\![Q]\!][x]/\widehat{I}^{Q}_{n}$; 
recall that $\mathfrak{G}^{Q}_{w} \in \bZ[Q_1, \ldots, Q_{n}][x_1, \ldots x_{n}] \subset \bZ[\![Q_1, \ldots, Q_{n}]\!][x_{1}, \ldots, x_{n}]$ for $w \in S_{n+1}$. 
For this purpose, it suffices to show that if $w \in S_{n+1}$ but $w \notin S_{n}$, then $\mathfrak{G}^{Q}_{w} = \widehat{Q}(\mathfrak{G}_{w})$ lies in the ideal of $\bZ[\![Q_1, \ldots, Q_{n}]\!][x_1, \ldots, x_{n}]$ generated by the images $\widehat{E}^{n}_{k} = \widehat{Q}(e^{n}_{k})$, $1 \leq k \leq n$. 
This can be shown by an argument which is similar to one for quantum Schubert polynomials in the proof of \cite[Lemma~5.7]{fgpqsp}, 
but which is based on \cite[Theorem~5.3]{lamqgp} instead of \cite[Proposition~5.4]{fgpqsp} (we refer the reader to \cite[Appendix~B]{lamqgp} for a proof in the torus-equivariant case); 
we also use the fact that for $w \in S_{n+1} \setminus S_{n}$, the associated classical Grothendieck polynomial $\mathfrak{G}_{w}$ lies in the ideal $I_{n}$ of $\mathbb{Z}[x] := \mathbb{Z}[x_1, \ldots, x_n]$ generated by the elementary symmetric polynomials $e^{n}_{k}$, $1 \leq k \leq n$, which follows since the ordinary $K$-theory $K(Fl_{n})$ of $Fl_{n}$ has a presentation of the form $\mathbb{Z}[x]/I_{n}$, and the classical Grothendieck polynomial $\mathfrak{G}_{w}$ represents the (opposite) Schubert class $[{\mathcal O}^{w}]$ in $K(Fl_{n})$ for each $w \in S_{n}$ under this presentation. 

We are now ready to state our main result of this paper, which settles the main conjecture (i.e., Conjecture~7.1) in \cite{lamqgp}.

\begin{thm}\label{qgroth} For each $w \in S_{n}$, the quantum Grothendieck polynomial $\mathfrak{G}^{Q}_{w}$ represents the (opposite) Schubert class $[{\mathcal O}^w]$ 
in $QK(Fl_{n}) = K(Fl_{n}) \otimes \mathbb{Z}[\![Q_{1}, \ldots, Q_{n-1}]\!]$. 
\end{thm}

\begin{proof}[Proof of Theorem~{\rm \ref{qgroth}}] 
We set $\bZ[Q] = \bZ[Q_{1}, \ldots, Q_{n-1}]$, $\bZ[\![Q]\!] = \bZ[\![Q_{1}, \ldots, Q_{n-1}]\!]$, $\bZ[Q]_{\mathrm{loc}} := \bZ[Q][(1-Q_1)^{-1}, \ldots, (1-Q_{n-1})^{-1}] \subset \bZ[\![Q]\!]$, and $\bZ[Q]_{\mathrm{loc}}[x] := \bZ[Q]_{\mathrm{loc}}[x_1, \ldots, x_n]$. 
Let $(I_{n}^{Q})_{\mathrm{loc}}$ denote the ideal of $\bZ[Q]_{\mathrm{loc}}[x]$ generated by the specialization at $Q_{n} = 0$ of the $\widehat{E}^{n}_{k}$ for $1 \leq k \leq n$. 
We know from \cite[Remark~3.27]{lamqgp} that the residue classes modulo $({I}^{Q}_{n})_{\mathrm{loc}}$ of the quantum Grothendieck polynomials $\mathfrak{G}^{Q}_{w}$, $w \in S_{n}$, form a $\bZ[Q]_{\mathrm{loc}}$-basis of the quotient ring $\bZ[Q]_{\mathrm{loc}}[x]/(I^{Q}_{n})_{\mathrm{loc}}$; we refer the reader to \cite[Appendix~B]{MNS} for a detailed proof of this fact in the torus-equivariant case.
Hence it follows that the residue classes modulo $\widehat{I}^{Q}_{n}$ of the $\mathfrak{G}^{Q}_{w}$, $w \in S_{n}$, form a $\bZ[\![Q]\!]$-basis of the quotient ring $\bZ[\![Q]\!][x]/\widehat{I}^{Q}_{n} \cong \bZ[\![Q]\!] \otimes_{\mathbb{Z}[Q]_{\mathrm{loc}}} (\bZ[Q]_{\mathrm{loc}}[x]/(I^{Q}_{n})_{\mathrm{loc}})$.
Also, we know that the (opposite) Schubert classes $[\mathcal{O}^{w}]$, $w \in S_{n}$, form a $\bZ[\![Q]\!]$-basis of $QK(Fl_{n}) = K(Fl_{n}) \otimes \bZ[\![Q]\!]$. Therefore, we can define a $\bZ[\![Q]\!]$-module isomorphism $\Phi$ from $\bZ[\![Q]\!][x]/\widehat{I}^{Q}_{n}$ to $QK(Fl_{n})$ by: $\Phi(\mathfrak{G}^{Q}_{w} \text{ mod } \widehat{I}^{Q}_{n}) = [\mathcal{O}^{w}]$ for $w \in S_{n}$. 

Here we consider the quotient ring $\bZ[x]/I_{n}$, where $I_{n}$ is, as above, the ideal of $\bZ[x]$ generated by the elementary symmetric polynomials $e^{n}_{k}$, $1 \leq k \leq n$.
We can easily verify by direct computation (the well-known fact) that the quotient ring $\bZ[x]/I_{n}$ is generated as an algebra over $\bZ$ by the residue classes modulo $I_{n}$ of the classical Grothendieck polynomials $\mathfrak{G}_{s_k} = 1 - \prod_{j=1}^{k} (1 - x_j)$ for $1 \leq k \leq n-1$.
Since the specialization at $Q_{1} = \cdots = Q_{n-1} = 0$ of the quantum Grothendieck polynomial $\mathfrak{G}^{Q}_{s_k}$ is identical to the classical Grothendieck polynomial $\mathfrak{G}_{s_k}$ for each $1 \leq k \leq n-1$ and the specialization at $Q_{1} = \cdots = Q_{n-1} = 0$ of $\widehat{E}^{n}_{k}$ is $e^{n}_{k}$ for each $1 \leq k \leq n$ by \cite[Proposition~3.22]{lamqgp}, it follows that the specialization at $Q_{1} = \cdots = Q_{n-1} = 0$ of the quotient ring $\bZ[\![Q]\!][x]/\widehat{I}^{Q}_{n}$ is isomorphic to the quotient ring $\bZ[x]/I_{n}$. 
Also, note that $\bZ[\![Q]\!][x]/\widehat{I}^{Q}_{n}$ is finitely generated as a module over $\bZ[\![Q]\!]$ since it is generated by $\mathfrak{G}^{Q}_{w} \text{ mod } \widehat{I}^{Q}_{n}$ for $w \in S_{n}$, as mentioned above; again we refer the reader to \cite[Appendix~B]{MNS} for the proof of the finite generation in the torus-equivariant case. 
Therefore, we can apply Nakayama's lemma (see \cite[Corollary~4.8]{E}) to deduce that the quotient ring $\bZ[\![Q]\!][x]/\widehat{I}^{Q}_{n}$ is generated as an algebra over $\bZ[\![Q]\!]$, which is a Noetherian integral domain such that the ideal $(Q_{1}, \ldots, Q_{n-1})$ is contained in the Jacobson radical, by the residue classes modulo $\widehat{I}^{Q}_{n}$ of the quantum Grothendieck polynomials $\mathfrak{G}^{Q}_{s_k}$, $1 \leq k \leq n-1$. 
Hence, from the coincidence between the quantum $K$-Chevalley formula (obtained from formula~\eqref{qkchev-f}) for opposite Schubert classes in the non-equivariant $QK(Fl_{n})$ and the Monk-type multiplication formula, together with the property above, for quantum Grothendieck polynomials, we deduce that the $\bZ[\![Q]\!]$-module isomorphism $\Phi$ is, in fact, a $\bZ[\![Q]\!]$-algebra isomorphism such that $\Phi(\mathfrak{G}^{Q}_{w} \text{ mod } \widehat{I}^{Q}_{n}) = [\mathcal{O}^{w}]$ for all $w \in S_{n}$. 
This completes the proof of the theorem.
\end{proof}

\begin{rema} {\rm 
Instead of the complete Noetherian integral domain $\bZ[\![Q]\!]$, we can use the Noetherian integral domain $\bZ[Q]_{\mathrm{loc}}$, which is a localization $S^{-1}(\bZ[Q])$ of $\bZ[Q]$ with respect to the multiplicative set $S := 1 + (Q_1, \ldots, Q_{n-1})$ (cf.~\cite[Appendix~A]{GMSZ}).
We know from \cite[Chapter 3, Exercise 2]{AM} that the ideal $S^{-1} (Q_1, \ldots, Q_{n})$ is contained in the Jacobson radical of $S^{-1}(\bZ[Q]) = \bZ[Q]_{\mathrm{loc}}$. 
Also, as mentioned in the proof above, the residue classes modulo $({I}^{Q}_{n})_{\mathrm{loc}}$ of the quantum Grothendieck polynomials $\mathfrak{G}^{Q}_{w}$, $w \in S_{n}$, form a $\bZ[Q]_{\mathrm{loc}}$-basis of the quotient ring $\bZ[Q]_{\mathrm{loc}}[x]/(I^{Q}_{n})_{\mathrm{loc}}$; in particular, the quotient ring $\bZ[Q]_{\mathrm{loc}}[x]/(I^{Q}_{n})_{\mathrm{loc}}$ is a finitely generated module over $\bZ[Q]_{\mathrm{loc}}$. 
Thus, we can apply Nakayama's lemma to the quotient ring $\bZ[Q]_{\mathrm{loc}}[x]/(I^{Q}_{n})_{\mathrm{loc}}$, and hence the same argument as in the proof above shows that the quotient ring $\bZ[Q]_{\mathrm{loc}}[x]/(I^{Q}_{n})_{\mathrm{loc}}$ is isomorphic to the subalgebra $K(Fl_{n}) \otimes \bZ[Q]_{\mathrm{loc}}$ of $QK(Fl_{n}) = K(Fl_{n}) \otimes \bZ[\![Q]\!]$. 
Here observe that by \cite[Remark on page 110]{AM}, the quotient ring $\bZ[Q]_{\mathrm{loc}}[x]/(I^{Q}_{n})_{\mathrm{loc}}$ can be thought of as a subalgebra of the quotient ring $\bZ[\![Q]\!][x]/\widehat{I}^{Q}_{n}$; in contrast, it is closely related to the finiteness result of Anderson-Chen-Tseng (see \cite[Proposition~9]{ACT}) that $K(Fl_{n}) \otimes \bZ[Q]_{\mathrm{loc}}$ is indeed a subalgebra of $QK(Fl_{n}) = K(Fl_{n}) \otimes \bZ[\![Q]\!]$.} 
\end{rema}

Theorem~\ref{qgroth} leads to an important application of quantum Grothendieck polynomials: computing the structure constants in $QK(Fl_n)$ with respect to the (opposite) Schubert basis. More precisely, the computation reduces to expanding the products of these polynomials in the basis they form. This is achieved by \cite[Algorithm~3.28]{lamqgp}, which can be easily programmed; see also \cite[Example~7.4]{lamqgp}. This application extends the similar one of Schubert polynomials, Grothendieck polynomials, and quantum Schubert polynomials, which was the main motivation for defining these polynomials. 

\subsection{The type $A$ quantum $K$-Chevalley coefficients}\label{sec:qk-coeff-type-A}

This section refers entirely to type $A_{n-1}$, more precisely to $QK(Fl_{n})$. Given a {\em degree} $d=(d_1,\ldots, d_{n-1})$, let $N_{s_k,w}^{v,d}$ be the coefficient of $Q_1^{d_1}\cdots Q_{n-1}^{d_{n-1}}[{\mathcal O}^v]$ in the expansion of $[{\mathcal O}^{s_k}]\cdot [{\mathcal O}^w]$, for $k\in I=\{1,\ldots,n-1\}$. Based on Theorem~\ref{qkchev} and results in \cite{lenfmp}, we describe more explicitly the quantum $K$-Chevalley coefficients $N_{s_k,w}^{v,d}$, where the index $k$ is fixed in this subsection. We need more background and notation. 

We start by recalling an explicit description of the edges of the quantum Bruhat graph on the Weyl group $W$ of type $A_{n-1}$, namely the symmetric group $S_{n}$. The permutations $w\in S_{n}$ are written in one-line notation $w=w(1)\cdots w(n)$. For simplicity, we use the same notation $(i,j)$ with $1\le i<j\le n$ for the positive root $\alpha_{ij}=\varepsilon_i-\varepsilon_j$ and the reflection $s_{\alpha_{ij}}$, which is the transposition $t_{ij}$ of $i$ and $j$; in particular, we write $w\cdot(i,j)$ for $wt_{ij}$, where $w\in S_n$. We need the circular order $\prec_i$ on $[n]:=\{1,\ldots,n\}$ starting at $i$, namely $i\prec_i i+1\prec_i\cdots \prec_i n\prec_i 1\prec_i\cdots\prec_i i-1$. It is convenient to think of this order in terms of the numbers $1,\ldots,n$ arranged clockwise on a circle. We make the convention that, whenever we write $a\prec b\prec c\prec\cdots$, we refer to the circular order $\prec=\prec_a$. We also consider the reverse circular order $\prec_i^r$ starting at $i$, namely $i\prec_i^r i-1\prec_i^r\cdots \prec_i^r 1\prec_i^r n\prec_i^r\cdots\prec_i^r i+1$, and use the same conventions.
	  
\begin{prop}[\cite{lenfmp}]\label{prop:quantum_bruhat_order_type_A}
For $w\in S_n$ and $1\leq i<j\leq n$, we have an edge $w \stackrel{(i,j)}{\longrightarrow} w\cdot(i,j)$ if and only if 
there is no $l$ such that $i<l<j$ and $w(i) \prec w(l) \prec w(j)$.
\end{prop}

It is proved in \cite[Corollary~15.4]{LP} that, given our fixed $k\in I$,
we have the following reduced $\varpi_k$-chain $\Gamma(\varpi_k)$: 
\begin{equation*}\begin{array}{lllll}
(&\!\!\!\!(k,k+1),&(k,k+2),&\ldots,&(k,n)\,,\\
&\!\!\!\!(k-1,k+1),&(k-1,k+2),&\ldots,&(k-1,n)\,,\\
&&&\ldots\\
&\!\!\!\!(1,k+1),&(1,k+2),&\ldots,&{(1,n)}\,\,)\,.
\end{array}\end{equation*}
Alternatively, we can use the  following reduced $\varpi_k$-chain $\Gamma'(\varpi_k)$: 
\begin{equation}\label{omegak-chain}\begin{array}{lllll}
(&\!\!\!\!(k,k+1),&(k-1,k+1),&\ldots,&(1,k+1)\,,\\
&\!\!\!\!(k,k+2),&(k-1,k+2),&\ldots,&(1,k+2)\,,\\
&&&\ldots\\
&\!\!\!\!(k,n),&(k-1,n),&\ldots,&{(1,n)}\,\,)\,.
\end{array}\end{equation}
We will also need a reduced $(-\varpi_k)$-chain $\Gamma(-\varpi_k)$, and we choose it to be just the reverse of $\Gamma(\varpi_k)$. Similarly, we define  another reduced $(-\varpi_k)$-chain $\Gamma'(-\varpi_k)$ as the reverse of $\Gamma'(\varpi_k)$.

Given $v,w\in S_n$, we write $v \smash{{}^\rightarrow_\rightarrow} w$ whenever there is a path from $v$ to $w$  in $\QB(S_n)$ with edges labeled by a subsequence of $\Gamma(\varpi_k)$. We also write $v\smash{{}^\rightarrow_\leftarrow} w$ (or $w\smash{{}^\leftarrow_\rightarrow} v$) whenever there is a path from $w$ to $v$  in $\QB(S_n)$ with edges labeled by a subsequence of $\Gamma(-\varpi_k)$.

We consider the following conditions on a pair $(v,w)$ of permutations in $S_n$; the first two appeared in~\cite[Section~4.1]{lenfmp}.

\noindent {\em Condition} A1. For any pair of indices $1\le i<j\le k$, both statements below are false: 
\begin{equation*}v(i)=w(j)\,,\;\;\;\;\;\;\;\;v(i)\prec w(j)\prec w(i)\,.\end{equation*}

\noindent {\em Condition} A2. For every index $1\le i\le k$, we have
\[w(i)=\min\,\{w(j) \mid i\le j\le k\}\,,\]
where the minimum is taken with respect to the circular order $\prec_{v(i)}$ on $[n]$ starting at $v(i)$.

\noindent {\em Condition} A$1'$. For any pair of indices $n\ge i>j\ge k+1$, both statements below are false: 
\begin{equation*}v(i)=w(j)\,,\;\;\;\;\;\;\;\;v(i)\prec^r w(j)\prec^r w(i)\,.\end{equation*}

\noindent {\em Condition} A$2'$. For every index $n\ge i\ge k+1$, we have
\[w(i)=\min\,\{w(j) \mid i\ge j\ge k+1\}\,,\]
where the minimum is taken with respect to the reverse circular order $\prec^r_{v(i)}$ on $[n]$ starting at $v(i)$.

 It is clear that Conditions~A1 and A2, respectively A$1'$ and A$2'$, are equivalent. We also consider similar conditions obtained by swapping the orders $\prec_{\cdot}$ and $\prec^r_{\cdot}$, which we label B1, B2, B$1'$, B$2'$, respectively.

\begin{prop}\label{ord-path} We have $v \smash{{}^\rightarrow_\rightarrow} w$ if and only if the pair $(v,w)$ satisfies Conditions~{\rm A1} and {\rm A$1'$}. Moreover, the corresponding path $v=v_0,\,v_1,\,\ldots,\,v_m=w$ in $\QB(S_n)$ (with edges labeled by a subsequence of $\Gamma(\varpi_k)$) is unique, and we have
\begin{align*}&v_0(i)\preceq v_1(i)\preceq\cdots\preceq v_m(i)\;\;\;\;\mbox{ for $\;1\le i\le k\,$,}\\
& v_0(i)\preceq^r v_1(i)\preceq^r\cdots\preceq^r v_m(i)\;\;\;\;\mbox{ for $\;n\ge i\ge k+1\,.$}\end{align*}
\end{prop}

\begin{proof} Given  $v \smash{{}^\rightarrow_\rightarrow} w$, the pair $(v,w)$ satisfies Condition~{\rm A1} by~\cite[Lemma~4.8~(1)]{lenfmp}. On another hand, the given path from $v$ to $w$ in $\QB(S_n)$ with edges labeled by a subsequence of $\Gamma(\varpi_k)$ can be transformed into a similar path with edges labeled by a subsequence of $\Gamma'(\varpi_k)$. Indeed, by comparing the structures of $\Gamma(\varpi_k)$ and $\Gamma'(\varpi_k)$, we can see that the first sequence of roots can be transformed into the second one by repeatedly swapping successive orthogonal roots; this implies that we can realize the mentioned transformation of paths in $\QB(S_n)$ by swapping successive commuting transpositions. We will prove Condition~A$1'$ based on the new path. The first part is immediate. Now assume for contradiction that $v(i)\prec^r w(j)\prec^r w(i)$ for $n\ge i>j\ge k+1$. Examining the sequence of transpositions that involve position $i$, we can see that one of them violates the criterion in Proposition~\ref{prop:quantum_bruhat_order_type_A}.

Now assume that the pair $(v,w)$ satisfies Conditions~{\rm A1} and {\rm A$1'$}. By \cite[Lemma~4.8~(2)]{lenfmp}, there exists a unique path $v=v_0,\,v_1,\,\ldots,\,v_m$ in $\QB(S_n)$ with $v_m(i)=w(i)$ for $i=1,\ldots,k$. Moreover, the stated property of the entries $v_j(i)$ for a fixed $i\in\{1,\ldots,k\}$ is part of the same lemma. Meanwhile, the case of $i\in\{k+1,\ldots,n\}$ is proved in a similar way, by using the the path labeled by a subsequence of $\Gamma'(\varpi_k)$ which is obtained from the above one. Finally, based on the first part of this proof, the pair $(v,v_m)$ satisfies Condition A$1'$, which further implies that $v_m=w$.
\end{proof}

\begin{rema}\label{two-alg}{\rm (1) Fix $v\in W=S_n$ and a representative $\sigma$ of a parabolic coset modulo $W_{I\setminus \{k\}}$. It is easy to see that there is a unique permutation $w\in\sigma W_{I\setminus \{k\}}$ such that the pair $(v,w)$ satisfies Conditions~A1 and {\rm A$1'$}. Indeed, the equivalent Conditions~A2 and A$2'$ lead to an algorithm which suitably reorders the entries $\sigma(1),\,\ldots,\,\sigma(k)$ and  $\sigma(k+1),\,\ldots,\,\sigma(n)$, respectively. More precisely, we iterate the construction of $w(i)$ given by Condition~A2 for $i=1,\ldots,k$, and the construction given by Condition~A$2'$ for $i=n,\ldots,k+1$. This reordering algorithm is explained in more detail in~\cite{lenfmp}; see Remark~4.5 and Example~4.6 therein.

(2) Given a pair $(v,w)$ which satisfies Conditions~{\rm A1} and {\rm A$1'$}, the construction of the unique path from $v$ to $w$ in Proposition~\ref{ord-path} is given by \cite[Algorithm~4.9]{lenfmp}; this is a greedy type algorithm.
}\end{rema}

We have the following corollary of Proposition~\ref{ord-path}.

\begin{cor}\label{ord-path-rev} We have $v \smash{{}^\rightarrow_\leftarrow} w$ if and only if the pair $(v,w)$ satisfies Conditions~{\rm B1} and {\rm B$1'$}. Moreover, the corresponding path $w=v_m,\,v_{m-1},\,\ldots,\,v_0=v$ in $\QB(S_n)$ (with edges labeled by a subsequence of $\Gamma(-\varpi_k)$) is unique, and we have
\begin{align}&v_0(i)\preceq^r v_1(i)\preceq^r\cdots\preceq^r v_m(i)\;\;\;\;\mbox{ for $\;1\le i\le k\,$,}\label{circ-entries1} \\
&v_0(i)\preceq v_1(i)\preceq\cdots\preceq v_m(i)\;\;\;\;\mbox{ for $\;n\ge i\ge k+1\,.$}\label{circ-entries2}\end{align}
\end{cor}

\begin{proof} We use the involution on $W$ given by $w\mapsto w^\circ:=w_\circ w$, which maps (quantum) edges of $\QB(W)$ to reverse (quantum) edges with the same labels, cf.~\cite[Proposition~4.4.1]{lnsumk1}. Therefore, we have $v \smash{{}^\rightarrow_\rightarrow} w$ if and only if $v^\circ \smash{{}^\rightarrow_\leftarrow} w^\circ$. Letting $i^\circ:=n+1-i$, we have  $w^\circ(i)=w(i)^\circ$, and we note that the involution $i\mapsto i^\circ$ on $[n]$ maps the order $\prec_i$ to $\prec_{i^\circ}^r$. Therefore, Conditions~A1 and A$1'$ correspond to Conditions~B1 and B$1'$ under this involution. We conclude that the statements of the corollary are translations of those in Proposition~\ref{ord-path}.
\end{proof}

\begin{rema}\label{two-alg-rev}{\rm By analogy with Remark~\ref{two-alg}~(1), Conditions B2 and B$2'$ lead to a corresponding reordering algorithm. Furthermore, there is an algorithm for constructing the unique path in $\QB(S_n)$ in Corollary~\ref{ord-path-rev} which is completely similar to \cite[Algorithm~4.9]{lenfmp}, cf. Remark~\ref{two-alg}~(2).}
\end{rema}

We are now ready to prove the main result of this section, which completely determines the quantum $K$-Chevalley coefficients $N_{s_k,w}^{v,d}$.

\begin{thm}\label{qk-coeff}
For $QK(Fl_{n})$, we always have $N_{s_k,w}^{v,d}\in\{0,\pm 1\}$. More precisely, for every $v$ and parabolic coset $\sigma W_{I\setminus \{k\}}$ not containing $v$, there are unique $d$ and $w\in \sigma W_{I\setminus \{k\}}$ (determined via the algorithms in Remark~{\rm \ref{two-alg-rev}} and {\rm \eqref{qkchev-f}}, cf. also Proposition~{\rm \ref{alg-deg}}) such that $N_{s_k,w}^{v,d}=\pm 1$ (the sign is as in {\rm \eqref{qkchev-f}}). Meanwhile, if $v\in\sigma W_{I\setminus \{k\}}$, then all the mentioned coefficients are $0$.
\end{thm}

\begin{proof} Fix $v$ and a parabolic coset $\sigma W_{I\setminus \{k\}}$. By~\eqref{qkchev-f} and Corollary~\ref{ord-path-rev}, a necessary condition to have $N_{s_k,w}^{v,d}\ne0$ for $w\in \sigma W_{I\setminus \{k\}}$ is that $w$ satisfies Conditions B1 and B$1'$. The unique such $w$ in $\sigma W_{I\setminus \{k\}}$ is constructed via the reordering algorithm mentioned in Remark~\ref{two-alg-rev}. Using Corollary~\ref{ord-path-rev} again, we know that there is a unique $w$-admissible subset $A\in{\mathcal A}(w,\Gamma(-\varpi_k))$ with ${\rm end}(w,A)=v$; this can be constructed via the second algorithm mentioned in  Remark~\ref{two-alg-rev}. If $v\in\sigma W_{I\setminus \{k\}}$, then we have $w=v$ and $A=\emptyset$, but the corresponding term does not appear in the right-hand side of the non-equivariant version of~\eqref{qkchev-f}. Otherwise, the corresponding degree $d=(d_1,\ldots,d_{n-1})$ and the sign of  $N_{s_k,w}^{v,d}=\pm 1$ are calculated based on~\eqref{qkchev-f}.
\end{proof}

We will now show that the degree $d$ in Theorem~\ref{qk-coeff} can be determined based on $v$ and $w$ only, that is, without constructing the respective path in $\QB(S_n)$ from $w$ to $v$. We use the notation $|\,\cdot\,|$ to indicate the cardinality of a set.

\begin{prop}\label{alg-deg} Given a pair $(v,w)$ satisfying Conditions {\rm B1} and {\rm B$1'$}, with $v\ne w$, the unique degree $d=(d_1,\ldots,d_{n-1})$ for which $N_{s_k,w}^{v,d}=\pm 1$ is expressed as follows{\rm :}
\[d_i=\casefour{|\{j \mid j\le i,\:v(j)<w(j)\}|}{i\in\{1,\ldots,k\}\,,}{|\{j \mid j> i,\:v(j)>w(j)\}|}{i\in\{k,\ldots,n-1\}\,.}\]
\end{prop}

\begin{proof} Consider $i,j\in\{1,\ldots,k\}$. It follows from Corollary~\ref{ord-path-rev}, and particularly~\eqref{circ-entries1}, that at most one of the roots $(j,l)$ in $\Gamma(-\varpi_k)$ (for $l>k$)  labels a quantum edge in the corresponding path from $w$ to $v$; moreover, this happens if and only if $v(j)<w(j)$. On the other hand, the simple root $\alpha_i=(i,i+1)$ appears in the decomposition of $(j,l)$ if and only if $j\le i$. This gives the formula for $d_i$ with $i\in\{1,\ldots,k\}$. The proof is completely similar for $i\in\{k,\ldots,n-1\}$, based on~\eqref{circ-entries2}.
\end{proof}

\begin{exa}{\rm Consider $v=12534$ in $S_5$, $k=2$, and $\sigma=34125$ in $W^{I\setminus\{2\}}$. The reordering algorithm in Remark~\ref{two-alg-rev} outputs $w=43215\in\sigma W_{I\setminus\{2\}}$. The second algorithm mentioned in Remark~\ref{two-alg-rev} determines the following path from $w$ to $v$ in $\QB(S_5)$; its edges are labeled by a subsequence of $\Gamma(-\varpi_2)$, which corresponds to an admissible subset $A$:
\[w=\tableau{{4}\\{3}\\ \\{2}\\{1}\\{5}}
\begin{array}{c} <\\ \xrightarrow{(1,5)} \end{array} 
\tableau{{5}\\ {3}\\ \\{2}\\ {1}\\{4}} 
\begin{array}{c} >\\ \xrightarrow{(1,4)} \end{array} 
\tableau{{1}\\{3}\\ \\{2}\\{5}\\{4}} 
\begin{array}{c} <\\ \xrightarrow{(2,4)} \end{array} 
\tableau{{1}\\{5}\\ \\{ 2}\\{3}\\{4}}
\begin{array}{c} >\\ \xrightarrow{(2,3)} \end{array} 
\tableau{{1}\\{2}\\ \\{5}\\{3}\\{4}}=v
\,.\]
Thus, we have ${\rm down}(w,A)=(1,4)+(2,3)$, and hence $d=(1,2,1,0)$; in fact, it is easier to determine $d$ based on Proposition~\ref{alg-deg}. Finally, we have $N_{s_2,w}^{v,d}=-1$. 
}
\end{exa}

\begin{rema} {\rm Analogous results to Proposition~\ref{ord-path} and the algorithms in Remark~\ref{two-alg} were given for type $C$ in~\cite{lenfmp}, and for types $B,\,D$ in~\cite{lascmbd}; they were used in connection with affine crystals and Macdonald polynomials. In addition, they can be used to obtain more explicit information about the quantum $K$-Chevalley coefficients in the respective types, by analogy with the above approach in type $A$. 
}\end{rema}

\subsection{Minimum and maximum degrees in type $A$}

Given the expansion of a product of two Schubert classes in quantum cohomology, there is interest in the following questions: are there minimum and maximum degrees, and do the degrees form intervals? These facts were proved to be true for type $A$ Grassmannians, where there is a single quantum variable $Q$ (see \cite{P}). For a partial flag manifold $G/P$, only the existence of a minimum power is known, which was proved to be the smallest degree of a rational curve between general translates of the corresponding Schubert varieties~\cite{bclecq}. Below we address these questions for the expansion of $[{\mathcal O}^{s_k}]\cdot [{\mathcal O}^w]$ in $QK(Fl_n)$, for a fixed $k\in I=\{1,\ldots,n-1\}$. 

We use the notation above; in particular, recall that $v\cdot(i,j)$ stands for $vt_{ij}$, where $v\in S_n$. We consider the set of roots in $\Gamma(-\varpi_k)$, for which we use the same notation, and we denote the corresponding linear order by $<$; in other words, we have  $(1,n)<\cdots<(1,k+1)<\cdots<(k,n)<\cdots<(k,k+1)$. We also consider the following partial order on these roots: $(a,b)\preceq (c,d)$ whenever $c\le a\le k<b\le d$; if $c<a\le k<b<d$, then we write $(a,b)\llcurly (c,d)$. 

Given a permutation $v\in S_n$, we define
\[\Gamma_v:=\{(i,j)\in \Gamma(-\varpi_k)\mid\, v\xrightarrow{(i,j)}v\cdot(i,j)\,\mbox{ is a quantum edge}\}\,.\]
For $A\subseteq\Gamma(-\varpi_k)$, we write $\Gamma_v^A:=\Gamma_v\cap A$. In particular, if 
\[A=\{(i,j)\in\Gamma(-\varpi_k)\mid (i,j)>(a,b)\}\,,\;\;\;\;B=\{(i,j)\in\Gamma(-\varpi_k)\mid (i,j)\llcurly(c,d)\}\,,\]
then we write
\[\Gamma_v^{>(a,b)}:=\Gamma_v^A\,,\;\;\;\;\Gamma_v^{\llcurly (c,d)}:=\Gamma_v^B\,,\;\;\;\;\Gamma_v^{>(a,b),\,\llcurly (c,d)}=\Gamma_v^{A,\llcurly (c,d)}:=\Gamma_v^{A\cap B}\,.\]
We use implicitly the quantum Bruhat graph criterion in Proposition~\ref{prop:quantum_bruhat_order_type_A}.

\begin{lem}\label{lemmm} The sets of roots $\Gamma_v$ and $\Gamma_v^{\llcurly (c,d)}$ have maxima and minima with respect to the partial order $\preceq$. 
\end{lem}

\begin{proof} Due to the structure of $\Gamma(-\varpi_k)$, it suffices to show that, if $(p,q)$ and $(r,s)$ belong to $\Gamma_v$ or $\Gamma_v^{\llcurly (c,d)}$, where $p<r\le k<q<s$, then so do $(p,s)$ and $(r,q)$. We have $v(p)>v(r)>v(q)>v(s)$, which implies the mentioned statement.
\end{proof}

We fix $w\in S_n$, and use Lemma~\ref{lemmm} in the following constructions based on $\Gamma_w$, which is assumed to be non-empty. Define the sequence of roots $(p_1,q_1),\,\ldots,\,(p_L,q_L)$ recursively by
\begin{equation}\label{constr-r}(p_1,q_1):=\mbox{max}_{\preceq}\Gamma_w\,,\;\;\;\;(p_{l+1},q_{l+1}):=\mbox{max}_{\preceq}\Gamma_w^{\llcurly(p_l,q_l)}\,,\end{equation}
for $l=1,\ldots,L-1$, where $\Gamma_w^{\llcurly(p_L,q_L)}=\emptyset$. Also, let $(r,s):=\mbox{min}_{\preceq}\Gamma_w$. For $1\le i<j\le n$, let
\[d(i,j):=(\underbrace{0,\,\ldots,\,0}_{i-1\;\,{\rm times}},\underbrace{1,\,\ldots,\,1}_{j-i\;\,{\rm times}},\underbrace{0,\,\ldots,\,0}_{n-j\;\,{\rm times}} )\]
be the degree corresponding to the root $\alpha_{ij}$.

We consider the non-zero degrees $d=(d_1,\ldots,d_{n-1})$ which occur in the expansion of $[{\mathcal O}^{s_k}]\cdot [{\mathcal O}^w]$, i.e., for which $N_{s_k,w}^{v,d}\ne 0$ for some $v\in S_n$. We assume that $\Gamma_w\ne\emptyset$, because otherwise there are no non-zero degrees. The following is the main result about the mentioned degrees.

\begin{thm}\label{maxmindeg}
Among the considered degrees, there is a maximum and a minimum one, with respect to the componentwise order. They are
\[d(p_1,q_1)+\cdots+d(p_L,q_L)\;\;\;\;\mbox{and}\;\;\;\;d(r,s)\,,\]
respectively.
\end{thm}

\begin{exa}\label{exdeg}{\rm Consider $n=4$, $w=4321$, and $k=2$. We have $\Gamma_w=\{(1,3),\,(1,4),\,(2,3),\,(2,4)\}=\Gamma(-\varpi_k)$. Thus, by Theorem~\ref{maxmindeg}, the maximum and minimum degrees are $(1,2,1)$ and $(0,1,0)$, respectively.}
\end{exa}

\begin{rema} {\rm (1) In quantum cohomology, the minimum degree in the corresponding Chevalley formula~\cite{fawqps} is the same as in quantum $K$-theory, while the maximum degree is $d(p_1,q_1)$, cf. Theorem~\ref{maxmindeg}.} 

{\rm (2) The non-zero degrees which occur in the expansion of $[{\mathcal O}^{s_k}]\cdot [{\mathcal O}^w]$ generally do not form an interval (between the minimum and the maximum one). Indeed, continuing Example~\ref{exdeg}, it is easy to see that $(0,2,0)$ is not a degree in the corresponding expansion.}
\end{rema}

To prove Theorem~\ref{maxmindeg}, we need the following lemmas.

\begin{lem}\label{onestep} Consider an edge $v\xrightarrow{(i,j)}v\cdot(i,j)$ in $\QB(S_n)$ labeled by $(i,j)\in\Gamma(-\varpi_k)$, and the subset $A$ of the roots in $\Gamma(-\varpi_k)$ occurring after some root $\alpha\ge(i,j)$. Then we have
\[\Gamma_{v\cdot(i,j)}^A\subseteq\Gamma_v^{A,\,\llcurly(i,j)}\;\;\;\;\mbox{or}\;\;\;\;\Gamma_{v\cdot(i,j)}^A\subseteq\Gamma_v^{A}\,,\]
depending on the edge labeled $(i,j)$ being or not being a quantum edge, respectively.
\end{lem}

\begin{proof}
Let $(a,b)\in \Gamma_{v\cdot(i,j)}^A$. If $(a,b)\llcurly(i,j)$, then the statement is obvious, and so we are left with the cases $i=a<b<j$ or $i<a<j\le b$. Assume first that $v\xrightarrow{(i,j)}v\cdot(i,j)$ is a quantum edge. The case  $i=a<b<j$ is easily ruled out, because it would imply $v(i)>v(j)>v(b)$, which contradicts the quantum property of the edge labeled by $(i,j)$. Similarly for the case $i<a<j=b$, and so we are left with the case $i<a<j<b$. But then the same quantum edge property implies $v(i)>v(a)>v(j)$, which makes it impossible to have $(a,b)\in \Gamma_{v\cdot(i,j)}$; so this case is ruled out, too. 

We now assume that $v\xrightarrow{(i,j)}v\cdot(i,j)$ is a cover of the Bruhat order. This means that there is no entry with value between $v(i)$ and $v(j)$ among $v(i+1),\ldots,v(j-1)$; we will use this fact implicitly below. We again start with the case $i=a<b<j$. We must have $v(b)<v(i)<v(j)$, which in turn implies $(a,b)\in \Gamma_{v}$. Similarly for the case $i<a<j=b$, and so we are left with the case $i<a<j<b$. The fact that  $(a,b)\in \Gamma_{v\cdot(i,j)}$ implies $v(b)<v(i)<v(j)<v(a)$, and then $(a,b)\in \Gamma_{v}$. 
\end{proof}

\begin{lem}\label{qbchain} Assume that we have a path in $\QB(S_n)$ starting at $w$ with edges labeled by roots $(i_1,j_1)<\cdots<(i_l,j_l)$ in $\Gamma(-\varpi_k)$, in this order. Then we have
\[\Gamma_{w\cdot(i_1,j_1)\cdots(i_l,j_l)}^{>(i_l,j_l)}\subseteq \Gamma_w^{>(i_l,j_l),\,\llcurly(i_s,j_s)}\;\;\;\;\mbox{or}\;\;\;\;
\Gamma_{w\cdot(i_1,j_1)\cdots(i_l,j_l)}^{>(i_l,j_l)}\subseteq \Gamma_w^{>(i_l,j_l)}\,,\]
depending on the path containing or not containing a quantum step, respectively, where in the first case $(i_s,j_s)$ is the label of the last quantum step.
\end{lem}

\begin{proof} We iterate the result in Lemma~\ref{onestep}. 
\end{proof}

\begin{proof}[Proof of Theorem~{\rm \ref{maxmindeg}}] By \eqref{qkchev-f}, the terms in the expansion of $[{\mathcal O}^{s_k}]\cdot [{\mathcal O}^w]$ are indexed by paths in $\QB(S_n)$ starting at $w$ with edges labeled by roots in $\Gamma(-\varpi_k)$, in the corresponding order $<$. Fix such a path, and let $(p_1',q_1')<\cdots<(p_m',q_m')$ be the labels of its quantum steps; then the degree corresponding to this path is $d(p_1',q_1')+\cdots+d(p_m',q_m')$. 

By Lemma~\ref{qbchain}, all $(p_i',q_i')$ are in $\Gamma_w$, and we have $(p_m',q_m')\llcurly\cdots\llcurly(p_2',q_2')\llcurly(p_1',q_1')$. By the construction~\eqref{constr-r}, we have $(p_1',q_1')\preceq(p_1,q_1)$. Combining the above facts, we have $(p_2',q_2')\llcurly (p_1,q_1)$, and by invoking again~\eqref{constr-r}, we derive $(p_2',q_2')\preceq(p_2,q_2)$. Continuing in the same way, we deduce $m\le L$, and $(p_i',q_i')\preceq(p_i,q_i)$ for all $i=1,\ldots,m$.

On another hand, it is easy to see that the path starting at $w$ with labels $(p_1,q_1),\,\ldots,\,(p_L,q_L)$, in this order, is indeed a path in $\QB(S_n)$ whose steps are all quantum ones; note that $(p_L,q_L)\llcurly\cdots\llcurly(p_2,q_2)\llcurly(p_1,q_1)$. This concludes the proof related to the maximum degree. The statement about the minimum degree is immediate based on the corresponding construction.
\end{proof}

\begin{cor}\label{max-deg} Among the terms in the quantum $K$-Chevalley formulas for the expansion of $[{\mathcal O}^{s_k}]\cdot [{\mathcal O}^w]$, where $k\in I$ is fixed and $w\in W$, there is a maximum degree (with respect to the componentwise order), namely
\[d_{\max}=(1,2,\ldots,\overline{k}-1,\underbrace{\overline{k},\ldots,\overline{k},}_{n+1-2\overline{k}\;\, times} \overline{k}-1,\ldots,2,1)\,,\]
where $\overline{k}:=\min(k,n-k)$. 
The maximum is attained (among other instances) for any $w,v$ with $w(i)=n+1-i$ and $v(i)=i$ for $i\le\overline{k}$ or $i>n-\overline{k}$, while $v(i)=w(i)$ for $\overline{k}<i\le n-\overline{k}$. 
The maximum total degree is $k(n-k)$.
\end{cor}

\begin{proof} If $w,v$ are as stated, then there is a path in $\QB(S_n)$ from $w$ to $v$ consisting only of quantum edges, which are labeled by $(1,n),\,\ldots,\,(\overline{k},n+1-\overline{k})$. On another hand, in the construction \eqref{constr-r} corresponding to the quantum $K$-Chevalley formula for $[{\mathcal O}^{s_k}]\cdot [{\mathcal O}^w]$, it is clear that we have $L\le\overline{k}$ and $(p_1,q_1)\preceq(1,n),\,\ldots,\,(p_L,q_L)\preceq(L,n+1-L)$. This implies the stated result.
\end{proof}

\begin{rema}
{\rm The maximum total degree in Corollary~\ref{max-deg} is equal to the (complex) dimension of the Grassmannian consisting of the $k$-dimensional subspaces in $\bC^n$, that is, the length of the maximum element in $W^{I\setminus\{k\}}$. }
\end{rema}

\end{document}